\documentclass[ejsv2,preprint,noshowframe,linksfromyear]{imsart}
\usepackage[authoryear,round]{natbib}
\usepackage{amsmath,amsthm,mathtools}
\usepackage{microtype,booktabs,array}
\usepackage{xcolor,graphicx}
\usepackage{refcount}
\usepackage[colorlinks=true,linkcolor=blue!45!black,citecolor=blue!45!black,urlcolor=blue!45!black]{hyperref}
\usepackage{bookmark}
\newcommand{\papertitle}{Minimax Slope Estimation and Support Recovery under Weak Gaussian Selection}
\hypersetup{pdftitle={Minimax Slope Estimation and Support Recovery under Weak Gaussian Selection},pdfauthor={Zeen Song}}
\startlocaldefs
\newtheorem{theorem}{Theorem}[section]
\newtheorem{lemma}[theorem]{Lemma}
\newtheorem{proposition}[theorem]{Proposition}
\newtheorem{corollary}[theorem]{Corollary}
\theoremstyle{definition}
\newtheorem{assumption}[theorem]{Assumption}
\theoremstyle{remark}
\newtheorem{remark}[theorem]{Remark}
\newcommand{\R}{\mathbb R}
\newcommand{\E}{\mathbb E}
\newcommand{\Prob}{\mathbb P}
\newcommand{\norm}[1]{\left\lVert#1\right\rVert}

\DeclareMathOperator{\TV}{TV}
\DeclareMathOperator{\diam}{diam}
\DeclareMathOperator{\Var}{Var}
\makeatletter
\newcommand{\aliaslabel}[2]{%
  \protected@write\@auxout{}{%
    \string\newlabel{#1}{{\getrefnumber{#2}}{\getpagerefnumber{#2}}{}{\getrefbykeydefault{#2}{anchor}{}}{}}}}
\makeatother

\endlocaldefs
\allowdisplaybreaks[2]
\setpkgattr{tablename}{size}{\upshape\scshape}
\setpkgattr{figurename}{size}{\upshape\scshape}

\begin{document}
\begin{frontmatter}
\pdfsubject{Preprint}
\title{\papertitle}
\runtitle{Slope estimation under weak Gaussian selection}
% Author name, affiliation and email supplied by the author on 2026-10-07.
\begin{aug}
\author[A]{\fnms{Zeen}~\snm{Song}\ead[label=e1]{zeensong17@gmail.com}}
\address[A]{School of Mathematics, Statistics and Mechanics,
Beijing University of Technology\printead[presep={,\ }]{e1}}
\runauthor{Z. Song}
\end{aug}

\begin{abstract}
In sample selection models, identifying a regression slope does not determine how accurately it can be estimated or distinguished from zero. We establish sharp uniform limits for outcome slopes in a fixed-dimensional Gaussian/probit selection model with unknown scale and an unknown covariate density bounded above and away from zero on a cube. We prove a uniform inverse bound between slope differences and observed-law Hellinger distance, allowing every nuisance parameter to differ between the two models. A matching path lies entirely within positive-correlation nonignorable selection. On parameter classes containing a feasible weak-selection neighborhood, the minimax squared slope risk is of order $n^{-2/3}$. For each prespecified minimum-signal sequence $b_n$, uniform exact support recovery is possible if and only if $n^{1/3}b_n\to\infty$, with rules allowed to depend on that sequence. The same estimator attains uniform root-$n$ slope accuracy on $\|\gamma_s\|\le M n^{-1/3}$ for each fixed finite $M\ge0$, including the exact zero-selection-slope layer. A testing-based construction establishes statistical attainability without estimating the covariate density; computational efficiency is not asserted.
\end{abstract}
\begin{keyword}[class=MSC]
\kwdgroup[type=primary]{\kwd{62D10}\kwd{62C20}}
\kwdgroup[type=secondary]{\kwd{62F12}\kwd{62J05}}
\end{keyword}
\begin{keyword}
\kwd{Nonignorable missingness}
\kwd{sample selection}
\kwd{Hellinger distance}
\kwd{minimax rates}
\kwd{support recovery}
\kwd{singular information}
\end{keyword}
\end{frontmatter}

\section{Introduction}
When an outcome is observed only for a nonrandom subset of a population, identifying its regression coefficients is only the first step toward reliable inference. A coefficient may be uniquely determined by the observed distribution and still be difficult to distinguish from zero. In Gaussian sample selection, changes in the outcome mean, noise scale, and selection mechanism can compensate one another. We ask how small an outcome slope can be estimated uniformly, and how large a nonzero slope must be for its presence to be recovered reliably.

The distinction between an exact parameter layer and its neighborhood is central. At zero marginal selection slope, selection shifts the respondent intercept but leaves its outcome slopes unchanged. Our full-model estimator has uniform root-$n$ slope accuracy on that layer, including zero correlation, and on $\|\gamma_s\|\le Mn^{-1/3}$ for every fixed finite $M\ge0$ (Corollary~\ref{cor:transition-upper}). Yet joint movement of selection slopes and correlation can conceal different outcome supports, giving worst-case squared slope risk of order $n^{-2/3}$ over a weak neighborhood. A related distinction between pointwise and local minimax rates appears in finite-mixture estimation \citep{HeinrichKahn2018}; here both outcome and nuisance parameters may move.

We study a fixed-dimensional Gaussian/probit experiment in which the noise scale is free and bounded, and the covariate density is unknown but bounded above and away from zero on a cube. Weakness here concerns the joint approach of $a=\sigma\rho$ and the marginal selection slope $g=\gamma_s$ to zero; the marginal response probabilities $\Prob(R=1\mid X)$ remain uniformly bounded away from zero. The main results follow a single chain. First, a uniform inverse bound controls the outcome-slope difference by the two-thirds power of the observed-law Hellinger distance, allowing all nuisance coordinates to move at both endpoints. Second, an explicit pair of paths with positive correlation matches this bound. Third, a testing-based estimator attains squared Hellinger risk at most $C/n$ without using the design density. These facts give the minimax squared slope rate and the exact signal order for uniform support recovery.

The matching paths remain strictly MNAR with the same correlation sign; their scale changes compensate second-order respondent variance. For each prespecified minimum-signal sequence $b_n$, uniform exact support recovery is possible precisely when $n^{1/3}b_n\to\infty$, with the rule allowed to use that sequence. This boundary concerns the observed experiment. Section~\ref{sec:estimation} gives the statistical construction, without asserting computational efficiency or adaptation to an unspecified separation sequence.

\paragraph{Related work.}
Singular sample-selection likelihoods and higher-order rates have a long history \citep[Sections 2.1, 5.1, and 6]{LeeChesher1986}; see also \citet[Sections 4.3--4.5]{RCBR2000}. \citet[Section 3.1 and Appendix A]{JinLee2018} use $r=\rho^3$ for singular likelihood estimation and group selection. At their fixed singular endpoint, non-intercept outcome slopes are root-$n$, the $n^{-1/3}$ rate concerns the variance parameter $\sigma^2$, and the penalized group consists of selection slopes and correlation. Our support lower bound instead compares two moving endpoints with different outcome supports. Cubic and centered skew-normal coordinates also precede this analysis \citep[Section 3.2 and Section 6.1]{HallinLey2014}. \citet{HoNguyen2019} study singularity structures and coordinate-dependent rates in finite mixtures, including skew-normal mixtures. These provide geometric context for the complete selection experiment studied here.

Hellinger geometry and testing-based estimation provide the methodological basis. \citet[Propositions 1--2]{BeckertMcFadden2010} relate moduli to uniform rates, with additional LAQ and estimator conditions for attainability. Our construction follows Birg\'e's $T$-estimator architecture \citep[Sections 4.2 and 5]{Birge2006} and local metric-dimension ideas \citep[Proposition 15 and Corollary 16]{BBS2017}. \citet{DoDoNguyen2025} derive conditional-law inverse bounds for regression mixtures, a different observed experiment. General conditional-distribution estimation with unknown design is treated by \citet{BaraudBirge2018}; conditional-integral corrections in candidate comparisons appear in \citet{Sart2017}. Here uniform density-coordinate and local packing bounds across the singular region make these testing ideas yield $C/n$ squared Hellinger risk, which combines with the slope geometry and matching path.

Related selection procedures include adaptive Lasso under positive definite information \citep[Section 4.3, C4 and Theorems 1--2]{Ogundimu2022} and Bayesian spike-and-slab selection \citep{IqbalEtAl2026}. \citet{LeebPotscher2008} emphasize that pointwise oracle properties need not control uniform risk; our support boundary quantifies this distinction over shrinking signals. Uniform inference near singularity is studied by \citet[Sections 3--4 and 6]{Bottai2003} for scalar models, including informative nonresponse, and by \citet[Theorem 2.4]{EkvallBottai2022} under their coverage conditions. \citet{HanMcCloskey2019} analyze extremum/Wald inference across identification strengths; their sample-selection example uses control-function moments. \citet{KhanNekipelov2024} study a semiparametric intercept under different design and tail conditions. These targets differ from the complete-law, shrinking-separation slope-support problem here. Our secondary confidence result is unconditional and simultaneous.

\citet{MaEtAl2026} study normal regression under a mixture of MAR and unrestricted MNAR response mechanisms. The regression result in the arXiv version \citep[version 2, Section 5, Theorem 17]{MaEtAl2026Preprint} uses a specified noise scale and gives a high-probability squared Euclidean error bound of order $1/\log n$ when the remaining parameters are fixed; its conditions also allow dimension growth. The positive MAR component imposes a lower bound on response probabilities conditional on the outcome and covariates. For nonzero correlation, the probit response probability in our model approaches zero in one outcome tail. The mechanism classes therefore do not support a direct rate ranking. Our targets are uniform squared slope risk and a shrinking-separation support boundary with free scale.

Section~\ref{sec:model} specifies the experiment. Sections~\ref{sec:geometry}--\ref{sec:main-support} develop the geometry, slope rates, and support boundary; Section~\ref{sec:restrictions} summarizes the restricted regimes. Section~\ref{sec:numerical} illustrates the separation mechanism. Appendices A--G contain all proofs and the detailed secondary results; Appendices~\ref{app:numerical}--\ref{app:additional-experiments} record the numerical protocols and complete results.

\section{The observed experiment and the parameter class}\label{sec:model}
Fix $d\ge1$, let $D=[-1,1]^d$, and let $\nu$ be normalized uniform measure on $D$. Write $z(x)=(1,x^T)^T$. The latent model is
\begin{equation}\label{eq:latent}
 Y=\beta_0+\beta_s^TX+\sigma\varepsilon,
 \qquad R=\mathbf1\{\gamma_0+\gamma_s^TX+U>0\},
\end{equation}
where $(\varepsilon,U)$ is jointly standard Gaussian with correlation $\rho$, independently of $X$. Set $m(x)=\beta_0+\beta_s^Tx$ and $t(x)=\gamma_0+\gamma_s^Tx$. The observation is $(X,O)$, with
\[
 O\in\mathcal O=\{\star\}\mathbin{\dot\cup}(\{1\}\times\R),
 \qquad O=\star\text{ if }R=0,\quad O=(1,Y)\text{ if }R=1.
\]
Thus $\star$ denotes nonresponse, not a zero outcome. Relative to $\mu$, the sum of unit mass at $\star$ and Lebesgue measure on the respondent branch, the conditional density is
\begin{align}
 Q_{\theta,0}(x)&=\Phi(-t(x)),\label{eq:atom}\\
 Q_{\theta,1}(y\mid x)&=\frac1\sigma\phi(e)
  \Phi\!\left(\frac{t(x)+\rho e}{\sqrt{1-\rho^2}}\right),
  \qquad e=\frac{y-m(x)}\sigma .\label{eq:respondent}
\end{align}
The respondent component integrates to $\Phi(t(x))$. Its normalization
$p_\theta(y\mid x,R=1)=Q_{\theta,1}(y\mid x)/\Phi(t(x))$ is a probability density. All statistics below use the atom and observed respondent outcomes only.

Let $\Theta$ be a fixed closed rectangular coefficient class for $\beta,\gamma$, together with
\[
 0<\sigma_-\le\sigma\le\sigma_+<\infty,
 \qquad |\rho|\le\rho_+<1.
\]
The intervals may be nonsymmetric. All bounds and $d$ are independent of $n$. The class may contain $\rho=0$; no information lower bound is imposed. The design law is $f\,d\nu$, where
\[
 \mathcal F_{\ell,u}=\{f:\textstyle\int f\,d\nu=1,
             \ 0<\ell\le f\le u<\infty\},\qquad \ell\le1\le u.
\]
The measurable density $f$ need not be smooth or known. An iid experiment uses one fixed $f$, and every pairwise comparison uses the same $f$. Write $P_\theta^f(dx,do)=f(x)Q_\theta(o\mid x)\nu(dx)\mu(do)$. We use the convention without a factor $1/2$:
\begin{align}
 H_f^2(\theta,\theta')&=\int_D f(x)\int_{\mathcal O}
       (\sqrt{Q_\theta}-\sqrt{Q_{\theta'}})^2\,d\mu\,d\nu,\label{eq:hellinger}\\
 h_\nu&=H_1,\qquad
 \sqrt\ell\,h_\nu\le H_f\le\sqrt u\,h_\nu.\label{eq:design-comparison}
\end{align}
In particular, $H_f\in[0,\sqrt2]$. A loss $H_f(\widehat\theta,\theta)$ evaluates the estimated conditional law with the actual common $f$; it does not mean that $\widehat\theta$ estimates the arbitrary marginal density of $X$.

Constants $C,c,A$ may change between displays and depend only on $d,\ell,u$, the coefficient intervals, $\sigma_-,\sigma_+$ and $\rho_+$, not on $n$, the shape of $f$, or the endpoints. Subclass constants may also depend on the stated fixed strength.

\citet[Theorem 2(c) and Example 4]{MiaoDingGeng2015} establish normal/probit identification and its relation to the Heckman parameterization. The following cumulant argument fixes our notation and is self-contained.

Write
\begin{equation}\label{eq:cumulants}
 \begin{split}
 a&=\sigma\rho,\quad g=\gamma_s,\quad t_0=\gamma_0,
 \quad\lambda=\phi/\Phi,\quad h=-\lambda',\quad\kappa=\lambda'',\\
 M(x)&=m(x)+a\lambda(t(x)),\quad
 V(x)=\sigma^2-a^2h(t(x)),\quad K_3(x)=a^3\kappa(t(x)).
 \end{split}
\end{equation}
These are the first three respondent cumulants; Appendix~\ref{app:geometry} derives them and proves $0<h<1$ and $\kappa>0$. The response probability identifies $t$ and, by the full-cube design, $\gamma$. Then $K_3/\kappa(t)$ identifies $a^3$ and its real cube root $a$; $M$ identifies $\beta$; and $V$ identifies $\sigma^2$, hence $\rho=a/\sigma$. This argument assumes, rather than validates, the exact Gaussian/probit model.

The support and its separated class are
\begin{equation}\label{eq:support-class}
 \begin{split}
 S(\theta)&=\{j\in\{1,\ldots,d\}:\beta_j\ne0\},\\
 \Theta(b)&=\{\theta\in\Theta:\ \beta_j=0\text{ or }|\beta_j|\ge b
                                  \text{ for every }j\ge1\}.
 \end{split}
\end{equation}
This definition includes empty support. Matching lower bounds need the following additional feasible-neighborhood condition; upper bounds do not.

\begin{assumption}[Natural weak neighborhood]\label{ass:natural}
There exist an interior scale $\sigma_0\in(\sigma_-,\sigma_+)$ and an interior selection intercept $t_0$ such that $\Theta$ contains a fixed neighborhood of
\[
 (\beta,\gamma,\sigma,\rho)=(0,(t_0,0),\sigma_0,0).
\]
In particular, the outcome coefficients and selection slopes admit two-sided neighborhoods of zero, and $\rho_+>0$. For a known-sign class, the intersection of this neighborhood with the corresponding closed correlation half-space suffices.
\end{assumption}
For sharpness, containment of the explicit paths below suffices. A smaller compact subclass need not admit a matching lower bound.

\section{Outcome-slope geometry}\label{sec:geometry}
\begin{theorem}[All-pairs cubic inverse modulus]\label{thm:modulus}
For every $f\in\mathcal F_{\ell,u}$ and every $\theta,\theta'\in\Theta$,
\begin{equation}\label{eq:main-modulus}
 \norm{\beta_s-\beta'_s}\le C H_f(\theta,\theta')^{2/3},
 \qquad H_f^2(\theta,\theta')\ge c\norm{\beta_s-\beta'_s}^{3}.
\end{equation}
All nuisance parameters may differ at the two endpoints. If one endpoint has $\norm{g}\ge g_0>0$ or $|a|\ge a_0>0$, then the stronger bound
$\norm{\beta_s-\beta'_s}\le C(g_0,a_0)H_f(\theta,\theta')$ holds against all competitors in $\Theta$.
\end{theorem}

The source of the exponent is explicit. With $H=H_f$, $v=|a-a'|$ and $G=\norm{g}$, observed moments and the strictly positive curvature of the inverse Mills ratio give
\begin{equation}\label{eq:geometry-mechanism}
 v^3\le CH,\qquad vG^2\le CH,\qquad
 \norm{\beta_s-\beta'_s}\le C(H+vG).
\end{equation}
Thus $vG\le\sqrt{v\,vG^2}\le CH^{2/3}$. At an endpoint with $G=0$, the slope bound is linear against every competitor, yielding Corollary~\ref{cor:zero-selection-slope}. Appendix~\ref{app:geometry} proves \eqref{eq:geometry-mechanism} with both endpoints unrestricted.

\begin{theorem}[Independent same-sign matching path]\label{thm:hardpath}
Under Assumption~\ref{ass:natural}, put $h_0=h(t_0)>0$ and, for $k=1,2$ and sufficiently small $r>0$, set
\begin{equation}\label{eq:hardpath}
 \begin{split}
 t_r(x)&=t_0+rx_1,\qquad L_r(x)=\lambda(t_0)+\lambda'(t_0)rx_1,\\
 a_k&=kr,\qquad \sigma_k^2=\sigma_0^2+k^2r^2h_0,
 \qquad \rho_k=a_k/\sigma_k,\qquad m_k(x)=(1-k)rL_r(x).
 \end{split}
\end{equation}
Both endpoints belong to the fixed class and have $\rho_k>0$. Their supports are $\varnothing$ and $\{1\}$, and uniformly over the common $f$,
\[
 \beta_{1,2}-\beta_{1,1}=h_0r^2,\qquad
 c r^3\le H_f(\theta_{1,r},\theta_{2,r})\le C r^3.
\]
Thus for each freely varying slope coordinate $j$ the modulus
$\sup_{H_f(\theta,\theta')\le\delta}|\beta_j-\beta'_j|$
has order $\delta^{2/3}$ as $\delta\downarrow0$.
\end{theorem}
The displayed path uses positive correlation. For a class restricted to negative correlation, Proposition~\ref{prop:sign} gives the corresponding reflected path. The scale changes compensate second-order respondent variance. Appendix~\ref{app:hardpath} proves the distance bound using transform uniqueness and a common third-derivative envelope, not merely matching moments, independently of the estimator and its density chart. With $r\asymp n^{-1/6}$, the path has scales
\[
 |a_k|\asymp\|g\|\asymp n^{-1/6},\qquad
 \|\beta_{s,2}-\beta_{s,1}\|\asymp n^{-1/3},\qquad
 H_f(\theta_{1,r},\theta_{2,r})\asymp n^{-1/2}.
\]

\begin{figure}[tbp]
 \centering
 \includegraphics[width=\linewidth]{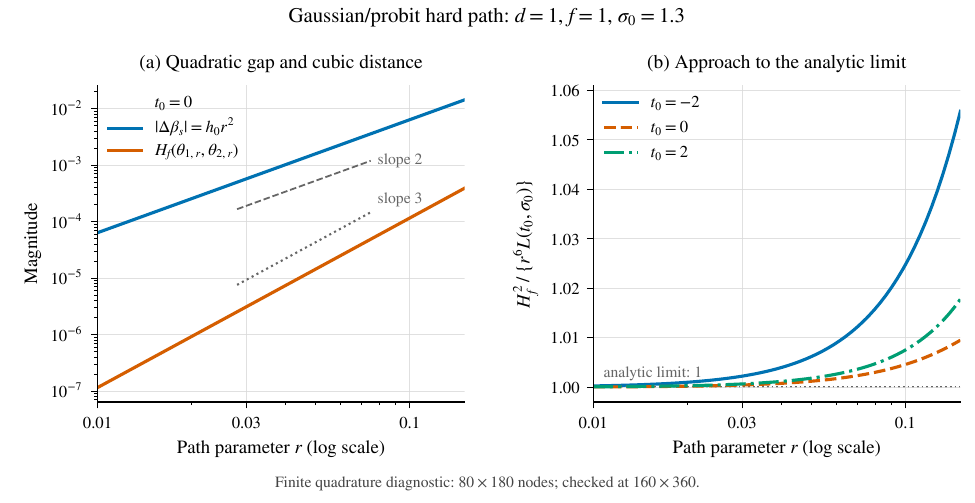}
 \caption{The same-sign path in \eqref{eq:hardpath}, with $d=1$, $f=1$ and $\sigma_0=1.3$. Left: the slope gap and Hellinger distance at $t_0=0$, with quadratic and cubic reference slopes. Right: $H_1^2/(r^6L)$ for three selection intercepts, where the positive coefficient $L=L(t_0,\sigma_0)$ is derived in \eqref{eq:path-limit}. The ratio approaches one as $r$ decreases. Numerical quadrature illustrates the analytic expansion; it supplies neither a proof nor an optimal global constant.}
 \label{fig:hard-path}
\end{figure}

\section{Observable estimation and slope rates}\label{sec:estimation}
We give the statistical construction before stating its risk bounds. The full proofs of the coordinate comparisons, local packing, and pairwise tests are in Appendices~\ref{app:chart}--\ref{app:testing}.

\subsection{Coordinates for the observed laws}\label{sec:finite-coordinates}
The singularity of the original parameters does not prevent a finite-dimensional description of distances between observed laws. In addition to $a,g,t_0$ from \eqref{eq:cumulants}, define
\begin{align}
 \eta_0&=\beta_0+a\lambda(t_0),&
 \eta_s&=\beta_s+a\lambda'(t_0)g,\notag\\
 v_0&=\sigma^2-a^2h(t_0),&
 q&=a^3,\qquad U_2=agg^T,\qquad Z=a^2g.\label{eq:overview-coordinates}
\end{align}
Here $\eta_0$ and $\eta_s$ are the value and gradient of the respondent mean at $x=0$, $v_0$ is its variance there, and $q$ records the cubic asymmetry. The coordinates $U_2$ and $Z$ retain the covariate dependence of the mean curvature and variance. Both are needed when the selection index varies.

Let $\operatorname{vech}$ list the upper triangular entries, including the diagonal, in a fixed order, and set
\begin{equation}\label{eq:overview-F}
 F(\theta)=(\gamma,\eta_0,\eta_s,v_0,q,\operatorname{vech}(U_2),Z)\in\R^K,
 \qquad K=3d+4+\frac{d(d+1)}2.
\end{equation}
Lemma~\ref{lem:embedding} proves the uniform comparisons
\begin{equation}\label{eq:overview-metric}
 c_E\|F(\theta)-F(\theta')\|_2
 \le h_\nu(\theta,\theta')
 \le C_E\|F(\theta)-F(\theta')\|_2.
\end{equation}
These hold across zero correlation as well as away from it. They concern conditional observed laws with the fixed reference design measure $\nu$; they do not embed or estimate the unrestricted marginal density $f$.

Set $\varepsilon=n^{-1/2}$. Lemma~\ref{lem:packing} constructs a deterministic, ordered set $\mathcal N_\varepsilon=\{\theta_1,\ldots,\theta_N\}\subset\Theta$ that is $\varepsilon$ separated and a $2\varepsilon$ cover in $h_\nu$. It starts with a sufficiently fine rectangular parameter grid and retains points in fixed order whenever their distance from every previously retained point exceeds $\varepsilon$. Every retained point is a legal original parameter. The coordinate comparison gives
\begin{equation}\label{eq:overview-packing}
 \#\{j:h_\nu(\theta_j,\vartheta)\le R\}
 \le (1+C_P R/\varepsilon)^K,
 \qquad \vartheta\in\Theta,\quad R\ge0.
\end{equation}
This local polynomial count is stronger than the existence of a finite global cover and is the input needed for the risk bound below.

\subsection{An observable finite-net estimator}\label{sec:estimator-construction}
The construction uses the $T$-estimator architecture of \citet{Birge2006}. For candidates $i,j$, write $Q_i=Q_{\theta_i}$ and define
\begin{align}
 \psi_{ij}(x,o)&=
 \frac{\sqrt{Q_j(o\mid x)}-\sqrt{Q_i(o\mid x)}}
      {\sqrt{Q_j(o\mid x)}+\sqrt{Q_i(o\mid x)}},\notag\\
 c_{ij}(x)&=\frac12\int_{\mathcal O}
       \{Q_i(o\mid x)+Q_j(o\mid x)\}\psi_{ij}(x,o)\,d\mu(o),\notag\\
 Z_{ij}(x,o)&=\psi_{ij}(x,o)-c_{ij}(x).\label{eq:overview-test}
\end{align}
The conditional integral is over the nonresponse atom and respondent line. It is calculated from the candidate laws at the observed $x$, so neither an unobserved outcome nor the design density enters the statistic. For a fixed orientation of each unordered pair, candidate $j$ defeats $i$ when $\sum_{r=1}^n Z_{ij}(X_r,O_r)>0$; the opposite sign favors $i$, and equality uses a predetermined winner. Conditional centering gives
\[
 \E_i^f Z_{ij}=-H_f^2(\theta_i,\theta_j)/2,
 \qquad \E_j^f Z_{ij}=H_f^2(\theta_i,\theta_j)/2
\]
for every allowed common $f$. Lemma~\ref{lem:pairtest} also controls the comparison when the truth is close to a candidate rather than exactly equal to it.

After all unordered pairs have been compared, set
\begin{equation}\label{eq:overview-estimator}
 D_i=\max\bigl(\{h_\nu(\theta_i,\theta_j):j\text{ defeats }i\}\cup\{0\}\bigr),
 \qquad
 \widehat\theta_n=\theta_{\min\operatorname*{arg\,min}_{1\le i\le N}D_i}.
\end{equation}
Thus the estimator minimizes the distance to its most distant defeating candidate, with a fixed rule for ties. It is Borel measurable and finite valued. The reference distances and candidate set are deterministic; the comparisons use only $(X_r,O_r)$. Exact conditional integrals and all pairwise comparisons specify a statistical procedure, but no bound on their numerical cost is claimed.

\subsection{Uniform risk bounds}

\begin{theorem}[Uniform statistical existence bound]\label{thm:estimator}
There is a finite-valued Borel estimator $\widehat\theta_n\in\Theta$, depending on the class bounds and on the observed data but not on $f$, and constants $A,c,C,w_*,v_*>0$ such that, for all $n\ge1$ and $w\ge w_*$,
\begin{align}
 \sup_{f,\theta}\Prob_{\theta}^{f,\otimes n}
  \{H_f(\widehat\theta_n,\theta)>w/\sqrt n\}&\le Ae^{-cw^2},\label{eq:main-risk-tail}\\
 \sup_{f,\theta}\E_{\theta}^{f,\otimes n}
       H_f^2(\widehat\theta_n,\theta)&\le C/n.\label{eq:main-risk}
\end{align}
After changing fixed constants, the first inequality also holds with $h_\nu$ in place of $H_f$. Moreover, for $v\ge v_*$,
\begin{equation}\label{eq:main-slope-tail}
 \sup_{f,\theta}\Prob_{\theta}^{f,\otimes n}
 \{\norm{\widehat\beta_s-\beta_s}>v n^{-1/3}\}\le Ae^{-cv^3}.
\end{equation}
Here and below the suprema range over $f\in\mathcal F_{\ell,u}$ and the indicated parameter class.
\end{theorem}
The theorem applies to the construction in Section~\ref{sec:estimator-construction}. The robust comparison bound and local count in \eqref{eq:overview-packing} control errors shell by shell around a candidate close to the truth. At radius $w/\sqrt n$, the count grows polynomially in $w$ while comparison errors decay exponentially in $w^2$. This yields the Hellinger tail without an additional logarithmic loss. The slope tail then follows from Theorem~\ref{thm:modulus}.

Write $\widehat\beta_{s,n}$ for the outcome-slope component of
$\widehat\theta_n$ in Theorem~\ref{thm:estimator}.
Proofs of these corollaries are given in Appendix~\ref{app:slope-corollaries}.

\begin{corollary}[Minimax outcome-slope squared risk]
\label{cor:minimax-slope-risk}
For every parameter class covered by Theorem~\ref{thm:estimator},
its estimator satisfies
\[
 \sup_{f\in\mathcal F_{\ell,u}}\sup_{\theta\in\Theta}
 \E_\theta^{f,\otimes n}
   \|\widehat\beta_{s,n}-\beta_s\|^2\le Cn^{-2/3},
 \qquad n\ge1.
\]
If Assumption~\ref{ass:natural} holds, there are fixed constants
$0<c\le C<\infty$ such that, for every $n\ge1$,
\begin{equation}
 c n^{-2/3}\le
 \inf_{\widetilde\beta_{s,n}}
 \sup_{f\in\mathcal F_{\ell,u}}\sup_{\theta\in\Theta}
 \E_\theta^{f,\otimes n}
   \|\widetilde\beta_{s,n}-\beta_s\|^2
 \le C n^{-2/3}.
 \label{eq:minimax-slope-risk}
\end{equation}
The infimum is over all Borel estimators with values in $\R^d$
based only on the observed sample, without $f$ supplied as an input.
For either known-sign class $\Theta\cap\{\rho\ge0\}$ or
$\Theta\cap\{\rho\le0\}$, the same two-sided rate holds when
the corresponding half-neighborhood in
Assumption~\ref{ass:natural} is feasible.
\end{corollary}

\begin{corollary}[Uniform slope rates on the zero-selection-slope layer]
\label{cor:zero-selection-slope}
Let $\Theta_0=\{\theta\in\Theta:\gamma_s=0\}$ be nonempty.
There is a constant $L_0>0$, depending only on the fixed class bounds,
such that, for every $\theta\in\Theta_0$, $\theta'\in\Theta$, and
$f\in\mathcal F_{\ell,u}$,
\begin{equation}
 \norm{\beta_s-\beta'_s}
 \le L_0h_\nu(\theta,\theta')
 \le L_0\ell^{-1/2}H_f(\theta,\theta').
 \label{eq:zero-selection-slope-modulus}
\end{equation}
The same full-model estimator as in Theorem~\ref{thm:estimator}
satisfies, for fixed $A_0,c_0,C_0,w_0>0$, all $n\ge1$, and
all $w\ge w_0$,
\begin{align}
 \sup_{f,\theta\in\Theta_0}
 \Prob_\theta^{f,\otimes n}
 \{\norm{\widehat\beta_{s,n}-\beta_s}>w/\sqrt n\}
 &\le A_0e^{-c_0w^2},
 \label{eq:zero-selection-slope-tail}\\
 \sup_{f,\theta\in\Theta_0}
 \E_\theta^{f,\otimes n}
       \norm{\widehat\beta_{s,n}-\beta_s}^2
 &\le C_0/n.
 \label{eq:zero-selection-slope-risk}
\end{align}
The constants are uniform over the entire layer, including $\rho=0$;
neither the estimator nor the competing endpoint is restricted to
$\Theta_0$.
\end{corollary}

\begin{corollary}[A transition upper bound near the zero-selection-slope layer]
\label{cor:transition-upper}
Under the assumptions of Theorem~\ref{thm:estimator}, let
$G=\|\gamma_s\|$ at the true parameter $\theta$. The same full-model
estimator satisfies, uniformly over $\theta\in\Theta$ and
$f\in\mathcal F_{\ell,u}$,
\begin{equation}
 \E_\theta^{f,\otimes n}
 \|\widehat\beta_{s,n}-\beta_s\|^2
 \le C\left[n^{-1}
       +\min\{G^2n^{-1/3},n^{-1}G^{-2}\}\right],\qquad G>0.
 \label{eq:transition-risk}
\end{equation}
For $G=0$, the bound is $C/n$. In particular, for every fixed
$0\le M<\infty$,
\begin{equation}
 \sup_{f\in\mathcal F_{\ell,u}}
 \sup_{\theta\in\Theta:\,\|\gamma_s\|\le Mn^{-1/3}}
 \E_\theta^{f,\otimes n}
 \|\widehat\beta_{s,n}-\beta_s\|^2
 \le \frac{C(1+M^2)}{n},\qquad n\ge1,
 \label{eq:shrinking-layer-risk}
\end{equation}
whenever the restricted class is nonempty. The estimator is not
restricted to this class and does not use $G$ or $M$.
\end{corollary}

If $g=0$, the respondent mean is $M(x)=\beta_0+a\lambda(t_0)+\beta_s^Tx$: selection shifts the intercept but leaves the slopes unchanged. Corollary~\ref{cor:transition-upper} extends root-$n$ slope accuracy to $\|g\|\le Mn^{-1/3}$ for fixed $M$, using the same unrestricted estimator. Its upper bound gives the global $n^{-1/3}$ slope scale at $\|g\|\asymp n^{-1/6}$, where the matching path conceals a slope change through joint nuisance variation. No matching lower bound for every drifting regime or maximal root-$n$ neighborhood radius is asserted.

\section{The exact support-separation boundary}\label{sec:main-support}
\begin{theorem}[Specified-sequence support frontier]\label{thm:support}
Assume the natural weak neighborhood. For each prespecified positive sequence $(b_n)$, there exists an observed-data support rule satisfying
\[
 \sup_{f\in\mathcal F_{\ell,u}}\sup_{\theta\in\Theta(b_n)}
 \Prob_{\theta}^{f,\otimes n}\{\widehat S_n(b_n)\ne S(\theta)\}\longrightarrow0
\]
if and only if $n^{1/3}b_n\to\infty$. For sufficiency one may take
\begin{equation}\label{eq:threshold-rule}
 \widehat S_n(b_n)=\{j\in\{1,\ldots,d\}:|\widehat\beta_j|>b_n/2\}.
\end{equation}
Its uniform error is at most $Ae^{-cnb_n^3}$ when $nb_n^3$ exceeds a fixed class constant. If $n^{1/3}b_n$ does not diverge, then along an infinite bounded subsequence the minimax support error is bounded below by a strictly positive constant. This lower bound holds even for a rule that knows any fixed allowed $f$.
\end{theorem}
The vector tail proves sufficiency; product affinity on the matching path proves necessity. Appendix~\ref{app:support} gives both arguments and explains the specified-sequence quantifier. The theorem concerns consistency rather than optimal leading constants.

\section{Restricted regimes}\label{sec:restrictions}
A known correlation sign preserves the weak-class boundary when scale is free and the corresponding half-neighborhood is feasible (Proposition~\ref{prop:sign}). The distinct known-scale problem is discussed in Remark~\ref{rem:known-scale}.

At truths with $\|g\|\ge g_0>0$ or $|a|\ge a_0>0$, the full-vector inverse bound is linear against all competitors (Proposition~\ref{prop:strong}). A fixed nonzero exclusion slope also yields root-$n$ target-slope accuracy, either in the known-zero model or through the full-model estimator (Proposition~\ref{prop:exclusion}).

For each prespecified $b_n$, fixed-strength support recovery is possible when $\sqrt n b_n\to\infty$; a matching converse requires a feasible perturbation from zero in an allowed target slope (Corollary~\ref{cor:regular-support}). Appendix~\ref{app:restrictions} gives the complete statements and proofs. The separate unconditional simultaneous confidence result is Corollary~\ref{cor:confidence}.

\section{Numerical illustration}\label{sec:numerical}
These calculations illustrate geometry and discrimination for specified endpoint pairs. They neither optimize over all competitors nor evaluate the full-model estimator. Uniform empirical risk, sharp asymptotic boundary constants, and certified quadrature errors are outside their scope; Appendices~\ref{app:numerical}--\ref{app:additional-experiments} give the protocols and complete results.

\subsection{Matched-gap geometry}
Figure~\ref{fig:geometry-comparison} contrasts the matching path with two paths on the exact zero-selection-slope layer. All three use $d=1$, $f=1$, $t_0=0$, and a common outcome-slope gap $\delta=h(0)r^2$. On the weak path, the nuisance coordinates vary as in \eqref{eq:hardpath}. On each comparison path, $g=0$, $\sigma=1.3$, and $a$ is fixed at either $0$ or $0.4$; only the outcome slope changes from $0$ to $\delta$. The latter choice retains nonzero outcome--selection dependence.

Consecutive log--log slopes of $H_f^2$ are close to $3$ on the weak path and $2$ on the zero-layer paths; the normalizations approach positive limits. Appendix~\ref{app:additional-experiments} reports the full ranges. These curves illustrate the distinct geometries in Theorem~\ref{thm:modulus} and Corollary~\ref{cor:zero-selection-slope}.

\begin{figure}[tbp]
\centering
\includegraphics[width=\linewidth]{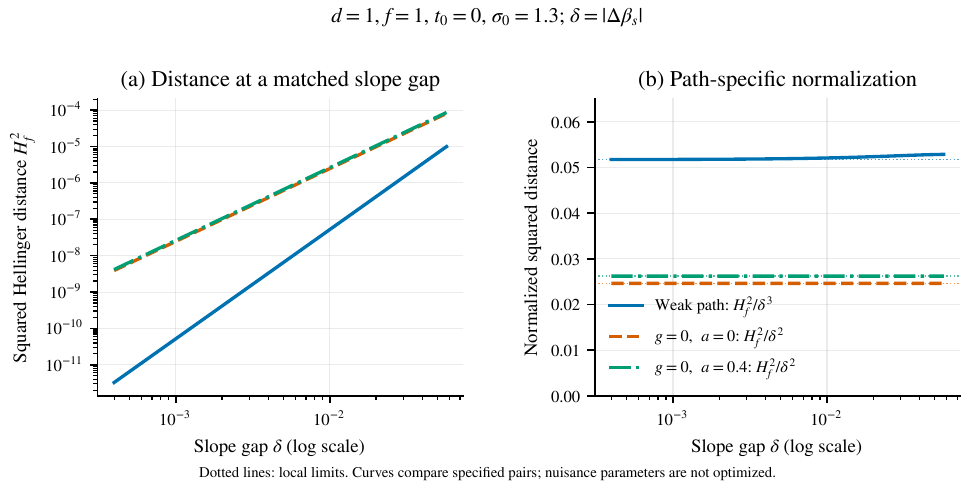}
\caption{Geometry at a matched slope gap. Left: observed-law squared Hellinger distance for the weak path and two exact $g=0$ paths. Right: $H_f^2/\delta^3$ on the weak path and $H_f^2/\delta^2$ on each zero-layer path; these different powers are indicated in the legend. Dotted lines are the corresponding local limits. Each curve uses all 20 fixed grid points, with identical nonresponse atoms within each pair.}
\label{fig:geometry-comparison}
\end{figure}

\subsection{Known-endpoint discrimination}
For the matching path, set $r_n=0.6(n/512)^{-\alpha}$ and
$b_n=h(0)r_n^2$. The choices $\alpha=1/4,1/6,1/12$ make $nb_n^3$
decrease, stay constant, or increase. At $n=32768$, the estimated
equal-prior errors are $0.423$, $0.286$, and $0.055$
(Figure~\ref{fig:two-point}). Appendix~\ref{app:numerical} specifies
the uniform design, 2000 repetitions per endpoint, and complete results.

The test receives both candidate endpoints. Persistent error illustrates the
support lower-bound obstruction even with that information; small error
establishes distinguishability only for the supplied pair.

Appendix~\ref{app:additional-experiments} repeats the grid for tilted and
edge-weighted designs in $\mathcal F_{0.4,2}$. It retains every scenario and
reports endpoint-specific, average, and maximum errors with Monte Carlo
intervals. Both examples show the same qualitative pattern.

\begin{figure}[tbp]
\centering
\includegraphics[width=\linewidth]{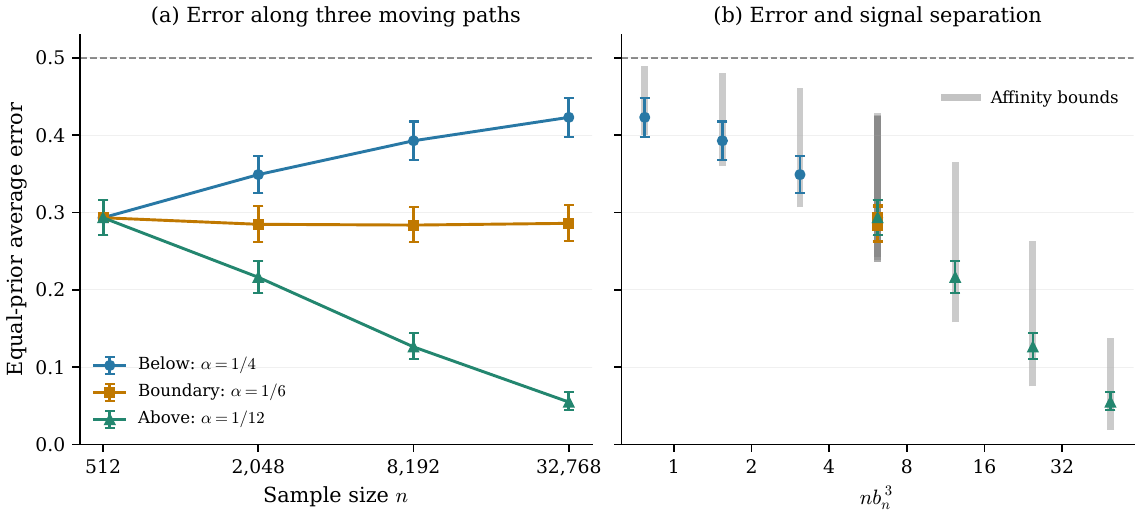}
\caption{Known-endpoint discrimination along three prespecified paths. Left: estimated average testing error as the sample size and parameter pair change together. Colored bars are conservative pointwise Monte Carlo intervals with coverage at least $95\%$. Right: all scenarios against $nb_n^3$; gray segments show the affinity bounds in \eqref{eq:two-point-affinity}, evaluated by quadrature. Boundary points share the same horizontal coordinate. The common $n=512$ pair is simulated once and reused. The dashed line marks error $1/2$.}
\label{fig:two-point}
\end{figure}

\section{Discussion}
The worst-case boundary persists with known design, alongside the root-$n$ slope accuracy on shrinking neighborhoods in Corollary~\ref{cor:transition-upper}. These slope bounds do not assert root-$n$ accuracy for the full parameter vector near the singularity. Sharp rates over all joint drifting sequences remain open.

The fixed-dimensional, full-cube design bounds and exact Gaussian/probit structure supply the projections, cumulants, and density envelopes used in the proofs. The distinction from semiparametric selection is substantive: \citet{HonoreHu2020} obtain partial-identification results without exclusion restrictions when the error distribution is unspecified. Here the exact model identifies the parameters, and the question is the uniform accuracy remaining within it.

The finite net and conditional integrals establish statistical attainability. A computational extension would need to control approximation and testing errors before transferring the risk bounds to a practical procedure; they do not automatically hold for maximum likelihood or adaptive Lasso. Other open problems are the sharp known-scale boundary (Remark~\ref{rem:known-scale}) and adaptation to an unspecified support-separation sequence.

\appendix
\section{Analytic proof of the inverse modulus}\label{app:geometry}
All norms in this appendix without another measure indicated are with respect to $\nu$. Choose finite $M_*,T,A$ from the parameter bounds so that $|m(x)|\le M_*$, $|t(x)|\le T$ and $|a|\le A$. In particular $p_*:=\Phi(-T)>0$. The latent representation of the respondent outcome is
\[
 Y=m(x)+aT_t+\sqrt{\sigma^2-a^2}W,
\]
where $T_t$ is standard normal conditioned to exceed $-t$ and $W$ is an independent standard normal. Completing the square gives the exact respondent cumulant generating function
\begin{equation}\label{eq:geom-cgf}
 K(s\mid x)=m(x)s+\frac{\sigma^2s^2}{2}
             +\log\Phi(t(x)+as)-\log\Phi(t(x)).
\end{equation}
Differentiating at zero proves \eqref{eq:cumulants}. The same Gaussian representation gives moments of every fixed order uniformly over the class.

\begin{lemma}[Strict inverse Mills curvature]\leavevmode\label{lem:curvature}
On $\R$, $0<h(t)<1$ and $\kappa(t)>0$. On every bounded interval they and the required derivatives are continuous and bounded, and
\[
 h_*:=\min_{|t|\le T}h(t)>0,\qquad
 \kappa_*:=\min_{|t|\le T}\kappa(t)>0.
\]
\end{lemma}
\begin{proof}
Put
\[
 A(t)=\frac{\Phi(t)}{\phi(t)}
       =\int_0^\infty e^{tv-v^2/2}\,dv,
 \qquad g_t(v)=A(t)^{-1}e^{tv-v^2/2}\mathbf1\{v>0\}.
\]
On a bounded $t$ interval each finite derivative under the integral has a polynomial times Gaussian integrable majorant. If $B$ has density $g_t$, then $\E B=A'/A$ and $\E B^2=A''/A$. Integration by parts yields $A'=tA+1$, hence
\[
 t+\lambda(t)=\E B>0,\qquad
 \lambda'(t)=-\lambda(t)\{t+\lambda(t)\}<0.
\]
The survival function is $S_t(v)=\Phi(t-v)/\Phi(t)$, with strictly increasing hazard $\lambda(t-v)$. For $x,y>0$, integration of that hazard shows
$S_t(x+y)<S_t(x)S_t(y)$. Tonelli's theorem therefore gives the strict inequality
\[
 \E B^2=2\int_0^\infty\!\int_0^\infty S_t(x+y)\,dx\,dy
        <2\left(\int_0^\infty S_t(v)\,dv\right)^2=2(\E B)^2.
\]
Twice differentiating $\lambda=1/A$ gives
\begin{equation}\label{eq:geom-positive}
 \kappa(t)=\lambda(t)\{2(\E B)^2-\E B^2\}>0.
\end{equation}
Also $B$ has the law of $T_t+t$. Differentiating its logarithmic normalizer twice, or using \eqref{eq:geom-cgf} for $T_t$, gives
$\Var(B)=\Var(T_t)=1+\lambda'(t)=1-h(t)>0$.
Together with $\lambda'<0$ this proves $0<h<1$. Finally $\Phi$ is positive everywhere; its quotient with $\phi$ is smooth. Continuity on the fixed interval now supplies positive minima and bounded derivatives. Strict positivity was proved analytically before taking minima.
\end{proof}

\begin{lemma}[Observed moments and the affine index]\label{lem:moments}
For arbitrary endpoints put $H=H_f(\theta,\theta')$. Then
\begin{equation}\label{eq:geom-moments}
 \norm{\Delta\gamma}+\norm{\Delta t}_\infty
 +\norm{\Delta M}_{L^2(\nu)}+\norm{\Delta V}_{L^2(\nu)}
 +\norm{\Delta K_3}_{L^2(\nu)}\le CH.
\end{equation}
Moreover $|a^3-a'^3|\le CH$ and $|a-a'|^3\le CH$.
\end{lemma}
\begin{proof}
Let $h_x$ be the conditional observed Hellinger distance and define
$m_j(x)=\int y^jQ_{\theta,1}(y\mid x)\,dy$ for $j=0,1,2,3$; thus $m_0=\Phi(t)$. These are observed submoments, not moments of an unobserved imputation. For $j\ge0$ let
$B_{2j}=\E(M_*+\sigma_+|W|)^{2j}$, interpreting $B_0=1$.
The respondent subdensity is bounded above by the unconditional outcome normal density. Cauchy--Schwarz gives
\begin{align*}
 |m_j-m_j'|^2
 &\le h_x^2\int y^{2j}(\sqrt{Q_{\theta,1}}+\sqrt{Q_{\theta',1}})^2\,dy
 \le4B_{2j}h_x^2,\\
 \norm{m_j-m_j'}_{L^2(\nu)}&\le2\sqrt{B_{2j}/\ell}\,H.
\end{align*}
Since both response probabilities exceed $p_*$ and the submoments are uniformly bounded, the same $L^2$ control holds for the normalized raw moments $r_j=m_j/m_0$. The cumulants
\[
 M=r_1,\qquad V=r_2-r_1^2,\qquad
 K_3=r_3-3r_1r_2+2r_1^3
\]
are Lipschitz polynomials on the bounded range of these raw moments. This proves their bounds without assuming that $Y$ is bounded; the sixth Gaussian moment suffices.

The derivative of $\Phi$ is at least $\phi(T)$ on the index interval, so
\[
 \norm{t-t'}_{L^2(\nu)}\le\frac{2H}{\sqrt\ell\,\phi(T)}.
\]
The affine Gram matrix is exactly
$G_\nu=\operatorname{diag}(1,1/3,\ldots,1/3)$. Thus
\begin{equation}\label{eq:geom-gram}
 \norm{\Delta\gamma}\le\sqrt3\norm{\Delta t}_{L^2(\nu)},
 \qquad \norm{\Delta t}_\infty\le\sqrt{1+3d}\norm{\Delta t}_{L^2(\nu)}.
\end{equation}
The actual-design Gram matrix satisfies $\ell G_\nu\preceq G_f\preceq uG_\nu$. No differentiability of $f$ is used.

For the cubic scalar coordinate, use the exact identity
\[
 (a^3-a'^3)\kappa(t)=K_3-K_3'-a'^3\{\kappa(t)-\kappa(t')\}.
\]
Its right side has $L^2(\nu)$ norm at most $CH$, by the preceding bounds and the bounded derivative of $\kappa$. Its left side has norm at least $|a^3-a'^3|\kappa_*$. Finally, with $v=|a-a'|$,
\[
 |a^3-a'^3|=v(a^2+aa'+a'^2),\qquad
 a^2+aa'+a'^2=\tfrac34(a+a')^2+\tfrac14(a-a')^2\ge v^2/4.
\]
This proves the remaining assertions for positive, negative, and zero correlations.
\end{proof}

\begin{lemma}[Quantitative nonaffine projection]\label{lem:projection}
Let $\Pi$ be orthogonal projection onto the affine functions in $L^2(\nu)$, let $g=\gamma_s$, and put $G=\norm{g}$. Then
\begin{align}
 \norm{(I-\Pi)\lambda(t)}_{L^2(\nu)}
    &\ge\frac{\kappa_*}{3\sqrt5\,d}G^2,\label{eq:geom-residual}\\
 \norm{c_s(t)}&\le L_*G,\qquad L_*:=\max_{|t|\le T}|\lambda'(t)|,\label{eq:geom-affine-slope}
\end{align}
where $c_s(t)$ is the slope vector of $\Pi\lambda(t)$. For the two endpoints of Lemma~\ref{lem:moments},
\begin{equation}\label{eq:geom-vg}
 |a-a'|G^2\le CH,\qquad
 \norm{\Delta\beta_s}\le CH+L_*|a-a'|G.
\end{equation}
\end{lemma}
\begin{proof}
Take $\psi(z)=(1-z^2)^2$ and $T_j(x)=\psi''(x_j)=12x_j^2-4$. The two boundary values of $\psi$ and $\psi'$ vanish, and direct integration under the normalized uniform measure gives
\[
 \E_\nu\psi(X_j)=8/15,\qquad
 \norm{T_j}_{L^2(\nu)}^2=64/5,\qquad
 \langle T_j,L\rangle_\nu=0\quad\text{for all affine }L.
\]
Two integrations by parts in $x_j$ give
\[
 \langle\lambda(t),T_j\rangle_\nu
  =g_j^2\E_\nu[\kappa(t)\psi(X_j)]\ge(8/15)\kappa_*g_j^2.
\]
Choose $j$ with $g_j^2\ge G^2/d$ and apply Cauchy--Schwarz to the residual to obtain \eqref{eq:geom-residual}. All integrations by parts take place under $\nu$, not under the possibly nonsmooth actual density $f$. In particular the corresponding actual-design residual satisfies
\[
 \inf_{L\text{ affine}}\norm{\lambda(t)-L}_{L^2(f\nu)}
 \ge\sqrt\ell\,\frac{\kappa_*}{3\sqrt5\,d}G^2.
\]
Subtracting $\lambda(t_0)$ and using the projection contraction and the affine Gram matrix yields
\[
 \norm{c_s(t)}\le\sqrt3\norm{\lambda(t)-\lambda(t_0)}_{L^2(\nu)}
 \le\sqrt3 L_*\norm{g^TX}_{L^2(\nu)}=L_*G.
\]
Now set
\begin{equation}\label{eq:geom-B}
 B(x)=M(x)-M'(x)-a'\{\lambda(t(x))-\lambda(t'(x))\}
     =m(x)-m'(x)+(a-a')\lambda(t(x)).
\end{equation}
Lemma~\ref{lem:moments} gives $\norm B_{L^2(\nu)}\le CH$. Its nonaffine projection and \eqref{eq:geom-residual} imply $|a-a'|G^2\le CH$; its affine projection, the Gram matrix, and \eqref{eq:geom-affine-slope} give the slope bound. This proves \eqref{eq:geom-vg} with both endpoints unrestricted.
\end{proof}

\begin{proof}[Proof of Theorem~\ref{thm:modulus}]
Put $v=|a-a'|$. Lemmas~\ref{lem:moments} and~\ref{lem:projection} give $v^3\le CH$ and $vG^2\le CH$. If $v>0$, then $vG=\sqrt{v(vG^2)}\le CH^{2/3}$; if $v=0$ the last term vanishes. Since $H\le\sqrt2$, the linear term is also bounded by a fixed multiple of $H^{2/3}$. This proves \eqref{eq:main-modulus}, including $H=0$.

If $G\ge g_0$, then $v\le Cg_0^{-2}H$. If $|a|\ge a_0$, then
\[
 a^2+aa'+a'^2=(a'+a/2)^2+3a^2/4\ge3a_0^2/4
\]
and the cubic-coordinate bound again gives $v\le CH$. Substitution into \eqref{eq:geom-vg} proves the asserted one-strong-endpoint result. In fact \eqref{eq:geom-B} and the variance equation then also control $\beta,\sigma,\rho$ at order $H$, as detailed in Appendix~\ref{app:restrictions}.
\end{proof}

\section{A feasible hard path and its root-density expansion}\label{app:hardpath}
\begin{proof}[Proof of Theorem~\ref{thm:hardpath}]
The parameters in \eqref{eq:hardpath} converge to the baseline of Assumption~\ref{ass:natural}. Thus a fixed $r_0>0$ makes both endpoints feasible for $0<r\le r_0$, with positive correlations less than $\rho_+$, interior scales, and
\[
 \sigma_k^2-a_k^2=\sigma_0^2-k^2r^2\{1-h_0\}\ge\sigma_0^2/2.
\]
The first endpoint has $m_1=0$. At the second,
\[
 \beta_{0,2}=-r\lambda(t_0),\qquad
 \beta_{1,2}=-r^2\lambda'(t_0)=h_0r^2,\qquad
 \beta_{j,2}=0\quad(j\ne1).
\]
This verifies the exact support and gap statements. It remains to prove the distance bound without assuming that low-order moment agreement alone implies root-density agreement.

Extend the displayed parameter formulas to $|r|\le r_0$ only for differentiation. Let $q_k(r;x,y)$ denote the unnormalized respondent component from \eqref{eq:respondent} along endpoint $k$, and write
\[
 q_k=\sigma_k^{-1}\phi(z_k)\Phi(w_k),\qquad
 z_k=(y-m_k)/\sigma_k,\qquad
 w_k=(t_r+\rho_k z_k)/\sqrt{1-\rho_k^2}.
\]
There is a common square-integrable envelope for the first three $r$ derivatives of $\sqrt{q_k}$. To verify the only tail issue, for $v\ge0$ use the integral in Lemma~\ref{lem:curvature} to obtain
\[
 A(-v)\ge\int_0^{1/(v+1)}e^{-vu-u^2/2}\,du
        \ge e^{-3/2}/(v+1).
\]
Hence $\lambda(w)\le C(1+|w|)$ for $w<0$; for $w\ge0$, use $\Phi(w)\ge1/2$. The identities
\[
 \lambda'=-\lambda(w+\lambda),\qquad
 \lambda''=\lambda(w^2-1)+3w\lambda^2+2\lambda^3
\]
show that these derivatives grow at most polynomially as well. Uniformly in $x,r,k$, all first three derivatives of $z_k,w_k$ are bounded by $C(1+|y|)$, because their scale denominators have a positive lower bound. The first three derivatives of $\ell_k=\log q_k$ are therefore bounded by a fixed polynomial in $|y|$. The exact square-root formulas are
\begin{align*}
 (\sqrt q)'&=\tfrac12\sqrt q\,\ell',\\
 (\sqrt q)''&=\sqrt q\{\ell''/2+(\ell')^2/4\},\\
 (\sqrt q)'''&=\sqrt q\{\ell'''/2+3\ell'\ell''/4+(\ell')^3/8\}.
\end{align*}
Since $\Phi(w_k)\le1$, $\sqrt{q_k}\le Ce^{-c y^2}$. A single function $C(1+|y|)^N e^{-c y^2}$ thus dominates these derivatives. Dominated convergence and the integral form of differentiation show that $r\mapsto\sqrt{q_k(r)}$ is $C^3$ into $L^2(f\,d\nu\,dy)$, uniformly over $f\le u$. The atom $\sqrt{\Phi(-t_r)}$ is smooth and identical at both endpoints. The same calculation for $q_k$ gives Gaussian envelopes for its density derivatives, including after multiplication by $e^{sy}$ for each fixed real $s$.

We next identify the zero, first, and second density derivatives. Put $M_r=rL_r$. Substituting \eqref{eq:hardpath} into the exact transform \eqref{eq:geom-cgf} and Taylor expanding $\log\Phi$ at $t_0$ gives, uniformly in $x$ and on bounded real $s$ intervals,
\begin{equation}\label{eq:path-transform}
 K_k(s\mid x)=M_rs+\tfrac12\sigma_0^2s^2
 +\kappa(t_0)r^3\left\{\frac{kx_1^2s}{2}
             +\frac{k^2x_1s^2}{2}+\frac{k^3s^3}{6}\right\}+O(r^4).
\end{equation}
Indeed, the linear term in $\log\Phi(t_0+r(x_1+ks))-\log\Phi(t_0+rx_1)$ is $kr\lambda(t_0)s$. Its second-order terms are
$\lambda'(t_0)r^2(kx_1s+k^2s^2/2)$. The affine mean in \eqref{eq:hardpath} reduces the first two location terms to $M_rs$, while $\sigma_k^2=\sigma_0^2-k^2r^2\lambda'(t_0)$ cancels the second-order variance term. The third-order difference is exactly the bracket in \eqref{eq:path-transform}. All derivatives of $\log\Phi$ needed for this bounded-argument expansion exist and are bounded.

The two unnormalized bilateral Laplace transforms are $\Phi(t_r)\exp\{K_k(s\mid x)\}$. By \eqref{eq:path-transform}, their zero, first, and second $r$ derivatives at zero agree for every real $s$. The Gaussian derivative envelopes proved above justify interchanging each of these derivatives with the $y$ integral. For each fixed $x$, a difference of these signed density derivatives has a Gaussian majorant. Its bilateral Laplace transform is therefore entire in $s$. Since that transform vanishes on the real axis, analytic uniqueness makes it vanish on the imaginary axis. Uniqueness of the Fourier transform of an integrable function makes the signed density itself zero almost everywhere. Thus the two density derivatives agree for orders zero through two. Their base density is the strictly positive
$\Phi(t_0)\sigma_0^{-1}\phi(y/\sigma_0)$, so the square-root chain rule makes the corresponding root-density derivatives equal as well. This argument uses the whole transform, not just finitely many moments.

In the Hilbert space of the observed square-root laws, let $S_k(r)$ denote endpoint $k$. Taylor's integral formula gives
\[
 S_k(r)=S_k(0)+rS_k'(0)+\tfrac12r^2S_k''(0)
              +\tfrac12\int_0^r(r-v)^2S_k'''(v)\,dv.
\]
The first three terms cancel between endpoints; the uniform third-derivative bound yields $\norm{S_2(r)-S_1(r)}\le Cr^3$. This is exactly $H_f\le Cr^3$. For the reverse bound, the third respondent cumulants differ by $7r^3\kappa(t_r)$. Lemma~\ref{lem:moments} and $\kappa\ge\kappa_*$ imply
\[
 7\kappa_*r^3\le\norm{K_{3,2}-K_{3,1}}_{L^2(\nu)}\le CH_f.
\]
This proves $H_f\asymp r^3$. Choosing $r=(\delta/C)^{1/3}$ gives a slope gap of order $\delta^{2/3}$ with distance at most $\delta$. Relabeling the cube coordinates gives every freely varying slope coordinate. Together with Theorem~\ref{thm:modulus}, this proves the modulus assertion. The path proof does not invoke the finite embedding or its statistical consequences.
\end{proof}

\paragraph{The coefficient used in Figure~\ref{fig:hard-path}.}
For the reference design $f=1$, the expansion also identifies the
leading squared-distance coefficient. Write $\kappa_0=\kappa(t_0)$,
$z=y/\sigma_0$ and $q_0(y)=\Phi(t_0)\sigma_0^{-1}\phi(z)$.
Let $\mathrm{He}_1(z)=z$, $\mathrm{He}_2(z)=z^2-1$ and
$\mathrm{He}_3(z)=z^3-3z$ be the first three probabilists' Hermite
polynomials. Using \eqref{eq:path-transform}, the difference of the
unnormalized transforms $\Phi(t_r)\exp\{K_k(s\mid x)\}$ has
third-order coefficient
$\kappa_0\{x_1^2s/2+3x_1s^2/2+7s^3/6\}$ times
$\Phi(t_0)\exp(\sigma_0^2s^2/2)$. Transform uniqueness, justified by the same
Gaussian derivative envelopes, identifies the corresponding density
coefficient. Since the lower-order density derivatives agree, the
square-root chain rule then gives
\begin{equation}\label{eq:path-root-coefficient}
 \frac{\sqrt{q_2(r;x,y)}-\sqrt{q_1(r;x,y)}}{r^3}
 \longrightarrow \sqrt{q_0(y)}\,\kappa_0
 \left\{\frac{x_1^2\mathrm{He}_1(z)}{4\sigma_0}
 +\frac{3x_1\mathrm{He}_2(z)}{4\sigma_0^2}
 +\frac{7\mathrm{He}_3(z)}{12\sigma_0^3}\right\}
\end{equation}
in $L^2(d\nu\,dy)$. This limit also follows from the $C^3$
root-density Taylor expansion above, so pointwise transform expansion
is not being used in place of norm convergence. The nonresponse
atoms agree for every $r$. Orthogonality under a standard Gaussian,
$\E\mathrm{He}_j(Z)^2=j!$, together with
$\int x_1^2d\nu=1/3$ and $\int x_1^4d\nu=1/5$, yields
\begin{equation}\label{eq:path-limit}
 L(t_0,\sigma_0):=
 \lim_{r\to0}\frac{H_1^2(\theta_{1,r},\theta_{2,r})}{r^6}
 =\Phi(t_0)\kappa_0^2
 \left\{\frac1{80\sigma_0^2}+\frac3{8\sigma_0^4}
                         +\frac{49}{24\sigma_0^6}\right\}>0.
\end{equation}
This is the asymptotic coefficient for this path and design, not an
optimal constant for the all-pairs modulus or support boundary.

\section{The cubic square-root density chart}\label{app:chart}

This appendix establishes the regularity needed to compare conditional laws
in centered coordinates.  The argument concerns normalized respondent
densities first; the response probability and missing-data atom are added
at the end.  Write
\[
 \phi_s(w)=(2\pi s)^{-1/2}\exp\{-w^2/(2s)\},\qquad s>0,
\]
where the subscript is a variance.  Let $T_t$ have the standard normal law
conditional on $T_t>-t$, and set $Z_t=T_t-\lambda(t)$.  Differentiating its
moment generating function gives
\[
 \mathbb{E}Z_t=0,\qquad \mathbb{E}Z_t^2=v(t)=1+\lambda'(t),
 \qquad \mathbb{E}Z_t^3=\kappa(t).
\]
In particular $v(t)>0$, because the truncated law is nondegenerate.
Lemma~\ref{lem:curvature} gives $h(t)=-\lambda'(t)>0$, so $v(t)<1$.

For $z=(M,V,a,t)$ with $s(z)=V-a^2v(t)>0$, define $p_z$ to be the density of
\begin{equation}\label{eq:chart-law}
 M+aZ_t+\sqrt{s(z)}W,\qquad W\sim N(0,1),\quad W\text{ independent of }T_t,
\end{equation}
and put $r_z=\sqrt{p_z}$.  The parametrization fixes respondent mean $M$
and variance $V$.  The real cubic coordinate is $q=a^3$.

\begin{lemma}[Cubic respondent chart]\label{lem:chart}
The map
\[
 \mathcal{S}(M,V,q,t)=r_{(M,V,\operatorname{cbrt}(q),t)}
\]
is continuously differentiable as an $L^2(dy)$-valued map on the open set
$V-|q|^{2/3}v(t)>0$.  At $q=0$ its derivative is
\begin{equation}\label{eq:chart-qzero}
 \partial_q\mathcal{S}(M,V,0,t)(y)
 =\frac{\kappa(t)}{12V^{3/2}}
 H_3\!\left(\frac{y-M}{\sqrt V}\right)\sqrt{\phi_V(y-M)},
 \qquad H_3(u)=u^3-3u.
\end{equation}
Fix finite $M_*,V_*,A_*,T_*$ and $\eta>0$.  On the set
\begin{equation}\label{eq:chart-bounded-domain}
 |M|\le M_*,\quad 0<V\le V_*,\quad |a|\le A_*,\quad |t|\le T_*,
 \quad V-a^2v(t)\ge\eta,
\end{equation}
the first derivatives in $(M,V,q,t)$ have uniformly bounded $L^2$ norms.
For any two endpoints in this set,
\begin{equation}\label{eq:chart-lipschitz}
 \|r_z-r_{z'}\|_2\le C\bigl(
 |M-M'|+|V-V'|+|a^3-a'^3|+|t-t'|\bigr),
\end{equation}
where $C$ depends only on the displayed bounds.
\end{lemma}

\begin{proof}
We give the differentiability and domination argument explicitly, since
pointwise smoothness alone does not justify differentiation of square
roots in $L^2$.

\paragraph{Fixed support and joint domination.}
Set $B=T_t+t$, so $B>0$, and write
\[
 z_t(b)=b-t-\lambda(t),\qquad w=y-M-a z_t(b).
\]
The joint density of $(B,Y)$ in \eqref{eq:chart-law} is
\begin{equation}\label{eq:chart-fixed-kernel}
 G_z(b,y)=\mathbf{1}_{\{b>0\}}
 \frac{\phi(b-t)}{\Phi(t)}\phi_{s(z)}(y-M-a z_t(b)),
 \qquad p_z(y)=\int_0^\infty G_z(b,y)\,db.
\end{equation}
Work first on an enlarged bounded set with $s\ge\eta/2$; derivatives at
its boundary are taken in a slightly larger open set with $s>\eta/4$.
All bounds below remain valid after this enlargement.  On $b>0$,
\begin{equation}\label{eq:chart-logjoint}
 \log G_z=-\log\Phi(t)-\frac{(b-t)^2}{2}
 -\frac12\log s-\frac{w^2}{2s}-\log(2\pi).
\end{equation}
Every parameter derivative of this expression of total order at most
three is a polynomial in $(b,y)$ of degree at most two, with uniformly
bounded coefficients.  Indeed, $w$ is affine in $(b,y)$, its parameter
derivatives are also affine, $s$ stays positive, and all needed
derivatives of $\lambda$ and $v$ are bounded on the enlarged $t$ interval.
It follows by repeated differentiation of the exponential that
$D^\alpha G_z/G_z$ is a polynomial of degree at most $2|\alpha|$ with
bounded coefficients, for $1\le|\alpha|\le3$.

For the joint Gaussian bound, let $B_0=b-t$ and
$Y_0=y-(M-a\lambda(t))$.  The exponent's quadratic form is
$B_0^2+(Y_0-aB_0)^2/s$.  If $s\le S_*$ on the enlarged set, then
\[
 B_0^2+Y_0^2
 \le(1+2A_*^2)B_0^2+2(Y_0-aB_0)^2
 \le\max\{1+2A_*^2,2S_*\}
 \left(B_0^2+\frac{(Y_0-aB_0)^2}{s}\right),
\]
with enlarged constants if necessary.  Both translation centers are
bounded.  The prefactor in \eqref{eq:chart-fixed-kernel} is bounded by
the positive lower bounds on $s$ and $\Phi(t)$.  Consequently, for fixed
$c,C>0$,
\begin{equation}\label{eq:chart-joint-envelope}
 G_z(b,y)\le Ce^{-c(b^2+y^2)},\qquad
 |D^\alpha G_z(b,y)|\le
 C(1+b+|y|)^6e^{-c(b^2+y^2)},\quad |\alpha|\le3.
\end{equation}
These are common integrable envelopes.  They justify differentiating
\eqref{eq:chart-fixed-kernel} under the integral through order three,
and show that the resulting derivatives of $p_z(y)$ are continuous.
Differentiating the Gaussian factor in $y$ up to order six is justified
by the same argument, with a larger fixed polynomial degree.

\paragraph{Conditional moments and square-root domination.}
An integrable bound on $D^\alpha p_z$ does not by itself control its
ratio to $p_z$ in the tails.  To obtain that control, put
\[
 \mu=M-a\lambda(t),\qquad D=s+a^2.
\]
Before truncation the pair $(B,Y)$ in the numerator of
\eqref{eq:chart-fixed-kernel} is Gaussian.  Conditional on $Y=y$ and
then on $B>0$, its first coordinate has the distribution
\begin{equation}\label{eq:chart-conditional-normal}
 B\mid Y=y\ \sim\ N\!\left(
 t+\frac{a}{D}(y-\mu),\frac{s}{D}\right)\ \text{conditioned on }B>0.
\end{equation}
The conditional mean is affine in $y$ with bounded coefficients, and
the conditional variance has positive lower and finite upper bounds.
For later use, if $B\sim N(m,\tau^2)$ conditioned on $B>0$, with
$0<\tau\le\tau_*$, then for every fixed integer $k\ge0$,
\begin{equation}\label{eq:chart-truncated-moments}
 \mathbb{E}B^k\le C_k(1+|m|^k).
\end{equation}
For $m\ge0$ the conditioning probability is at least $1/2$, so the
claim follows from ordinary Gaussian absolute moments.  For $m<0$ the
conditional density on the positive half-line is a half-normal density
tilted by the decreasing function $b\mapsto e^{mb/\tau^2}$.  This tilt
decreases every increasing moment: for increasing $F$ and decreasing
$W$, two independent half-normal variables $B,B'$ satisfy
\[
 2\operatorname{Cov}(F(B),W(B))
 =\mathbb{E}\{(F(B)-F(B'))(W(B)-W(B'))\}\le0.
\]
Taking $F(b)=b^k$ proves the bound by the half-normal moment.  Applied
to \eqref{eq:chart-conditional-normal}, this yields
\[
 \mathbb{E}_z[(1+B+|y|)^k\mid Y=y]\le C_k(1+|y|)^k.
\]
The differentiation already justified above and conditional expectation
therefore give, for $1\le|\alpha|\le3$,
\begin{equation}\label{eq:chart-density-ratios}
 \frac{D^\alpha p_z(y)}{p_z(y)}
 =\mathbb{E}_z\!\left[
 \frac{D^\alpha G_z(B,y)}{G_z(B,y)}\,\middle|\,Y=y\right],
 \qquad
 \left|\frac{D^\alpha p_z(y)}{p_z(y)}\right|
 \le C(1+|y|)^{2|\alpha|}.
\end{equation}
The marginal density is strictly positive, as is also seen from
\begin{equation}\label{eq:chart-explicit-marginal}
 p_z(y)=\frac{\phi_D(y-\mu)}{\Phi(t)}
 \Phi\!\left(\frac{t+a(y-\mu)/D}{\sqrt{s/D}}\right),
 \qquad 0<p_z(y)\le Ce^{-cy^2}.
\end{equation}
The last bound uses the uniform variance and location bounds for
$\phi_D(y-\mu)$, together with the lower bound for $\Phi(t)$.

For $r=\sqrt p$, direct differentiation gives
\begin{align}
 r_a&=\frac{p_a}{2\sqrt p},\qquad
 r_{aa}=\frac{p_{aa}}{2\sqrt p}-\frac{p_a^2}{4p^{3/2}},
 \label{eq:chart-root-first-two}\\
 r_{aaa}&=\frac{p_{aaa}}{2\sqrt p}
 -\frac{3p_a p_{aa}}{4p^{3/2}}+\frac{3p_a^3}{8p^{5/2}}.
 \label{eq:chart-root-third}
\end{align}
For $\xi=M,V,t$, $r_\xi=p_\xi/(2\sqrt p)$.  Combining
\eqref{eq:chart-density-ratios} with \eqref{eq:chart-explicit-marginal}
shows that one function
\begin{equation}\label{eq:chart-Ltwo-envelope}
 \mathcal{G}(y)=C(1+|y|)^6e^{-cy^2/2}\in L^2(dy)
\end{equation}
dominates $|r|,|r_a|,|r_{aa}|,|r_{aaa}|,|r_M|,|r_V|,|r_t|$
throughout the enlarged set.  Each derivative is pointwise continuous,
so domination by $\mathcal{G}^2$ proves its $L^2$ continuity.  Applying
the pointwise fundamental theorem of calculus along a sufficiently
small parameter segment, followed by the same domination, shows that
these are the actual $L^2$ derivatives.  In particular $r$ is three
times continuously differentiable in $a$, with the indicated
derivatives jointly continuous in the other parameters.

\paragraph{Three derivatives and the two cancellations.}
Keep $M,V,t$ fixed and abbreviate $Z=Z_t$, $v=v(t)$ and
$k=\phi_{V-a^2v}(y-M-aZ)$.  Write $k^{(j)}=\partial_y^j k$.
The heat equation $\partial_s\phi_s=\tfrac12\partial_y^2\phi_s$
implies $\partial_a k=-Zk^{(1)}-avk^{(2)}$.  Repeated differentiation,
with the integrals justified by \eqref{eq:chart-joint-envelope}, gives
\begin{align}
 p_a&=\mathbb{E}[-Zk^{(1)}-avk^{(2)}],\label{eq:chart-pa}\\
 p_{aa}&=\mathbb{E}[(Z^2-v)k^{(2)}+2avZk^{(3)}
                         +a^2v^2k^{(4)}],\label{eq:chart-paa}\\
 p_{aaa}&=\mathbb{E}\bigl[-(Z^3-3vZ)k^{(3)}
                  -3av(Z^2-v)k^{(4)}\notag\\
 &\hspace{37mm}{}-3a^2v^2Zk^{(5)}-a^3v^3k^{(6)}\bigr].
 \label{eq:chart-paaa}
\end{align}
At $a=0$ the kernel is independent of $Z$, whence
\begin{equation}\label{eq:chart-cancellations}
 p_a(0)=-\mathbb{E}Z\,\phi_V'=0,\qquad
 p_{aa}(0)=(\mathbb{E}Z^2-v)\phi_V''=0,
 \qquad r_a(0)=r_{aa}(0)=0.
\end{equation}
These are identities for derivatives of the full density and its
square root, not conclusions drawn solely from equality of two moments.
The variance compensation in \eqref{eq:chart-law} is essential.
Moreover,
\[
 \log\mathbb{E}e^{uZ}
 =\frac{u^2}{2}+\log\Phi(t+u)-\log\Phi(t)-u\lambda(t),
\]
so $\mathbb{E}Z^3=\kappa(t)$ and
\begin{equation}\label{eq:chart-third-at-zero}
 p_{aaa}(0)=-\kappa(t)\phi_V'''(y-M),\qquad
 r_{aaa}(0)=\frac{\kappa(t)}{2V^{3/2}}
 H_3\!\left(\frac{y-M}{\sqrt V}\right)\sqrt{\phi_V(y-M)}.
\end{equation}

\paragraph{The derivative through zero.}
If $(M,V,a,t)$ is feasible, its segment to $a=0$ is feasible with at
least the same variance margin, because
\begin{equation}\label{eq:chart-a-segment}
 V-u^2v(t)\ge V-a^2v(t),\qquad |u|\le|a|.
\end{equation}
Thus the Hilbert-space fundamental theorem of calculus and
\eqref{eq:chart-cancellations} give, for either sign of $a$,
\begin{equation}\label{eq:chart-integral-derivative}
 r_a(a)=a^2\int_0^1(1-u)r_{aaa}(ua)\,du.
\end{equation}
If $B$ is the uniform $L^2$ bound on $r_{aaa}$, this implies
$\|r_a(a)\|_2\le Ba^2/2$.  For $q\ne0$ and $a=\operatorname{cbrt}(q)$,
\begin{equation}\label{eq:chart-q-derivative}
 \mathcal{S}_q(M,V,q,t)=\frac{r_a(a)}{3a^2}
 =\frac13\int_0^1(1-u)r_{aaa}(ua)\,du,
 \qquad \|\mathcal{S}_q\|_2\le B/6.
\end{equation}
Taylor's integral formula also gives
\begin{equation}\label{eq:chart-zero-difference-quotient}
 \frac{\mathcal{S}(M,V,q,t)-\mathcal{S}(M,V,0,t)}q
 =\frac12\int_0^1(1-u)^2r_{aaa}(ua)\,du
 \longrightarrow\frac16r_{aaa}(0)
\end{equation}
in $L^2$, from both signs of $q$.  This proves
\eqref{eq:chart-qzero}.  Joint continuity of $r_{aaa}$ and its common
envelope show that \eqref{eq:chart-q-derivative} extends continuously
to this value as $(M,V,q,t)$ approaches any point with $q=0$.
The ordinary partial derivatives in $M,V,t$ are continuous by
\eqref{eq:chart-Ltwo-envelope} and continuity of the real cube root.
Continuous partial derivatives on this open domain give joint
$L^2$-valued $C^1$ regularity.  Integrating the bounded $q$ derivative
at fixed $M,V,t$ also gives a Lipschitz bound across zero: when the
endpoint signs differ, $|q-q'|=|q|+|q'|$ accounts for both segments.

\paragraph{Local interpolation and the endpoint bound.}
It remains to justify interpolation when all four coordinates move.
Let
\[
 \Delta=|M-M'|+|V-V'|+|q-q'|+|t-t'|,
 \quad v_* =\max_{|t|\le T_*}v(t),
 \quad L_v=\max_{|t|\le T_*}|v'(t)|.
\]
Interpolate linearly in $(M,V,q,t)$, using subscript $u\in[0,1]$.
Since $|q_u|\le A_*^3$ and
$\bigl||q_u|^{2/3}-|q|^{2/3}\bigr|\le|q_u-q|^{2/3}$,
\begin{equation}\label{eq:chart-interpolation-margin}
 V_u-|q_u|^{2/3}v(t_u)
 \ge\eta-(1+A_*^2L_v)\Delta-v_*\Delta^{2/3}.
\end{equation}
An explicit sufficient neighborhood is
\begin{equation}\label{eq:chart-interpolation-radius}
 \delta_0=\min\left\{1,\frac{\eta}{4(1+A_*^2L_v)},
                      \left(\frac{\eta}{4v_*}\right)^{3/2}\right\}>0.
\end{equation}
When $\Delta\le\delta_0$, the whole segment has convolution variance
at least $\eta/2$.  All its ordinary parameter bounds are retained,
and its auxiliary segments to $a=0$ retain the same margin by
\eqref{eq:chart-a-segment}.  The uniformly bounded first derivatives
can therefore be integrated on this segment to prove
\eqref{eq:chart-lipschitz}.  When $\Delta>\delta_0$, two normalized
densities have square-root distance at most $\sqrt2$, and
$\sqrt2\le(\sqrt2/\delta_0)\Delta$ gives the same result with a larger
constant.  A globally feasible straight segment between distant
endpoints is neither needed nor asserted.
\end{proof}

For parameters of the regression model, the chart coordinates are
\[
 M=m+a\lambda(t),\qquad V=\sigma^2-a^2h(t),\qquad q=a^3,
\]
and their convolution variance satisfies the uniform bound
\begin{equation}\label{eq:chart-model-margin}
 V-a^2v(t)=\sigma^2-a^2
 =\sigma^2(1-\rho^2)
 \ge\sigma_-^2(1-\rho_+^2)>0.
\end{equation}
All other chart bounds follow from the fixed coefficient and scale
bounds.  In the direct sum $\mathbb{R}\oplus L^2(dy)$ the square root
of the full conditional observed law is
\begin{equation}\label{eq:chart-observed-root}
 \left(\sqrt{\Phi(-t)},\ \sqrt{\Phi(t)}\,
       \mathcal{S}(M,V,q,t)\right).
\end{equation}
Both scalar square-root probabilities have bounded derivatives on the
fixed $t$ interval, and $\|\mathcal{S}\|_2=1$.  Thus
\eqref{eq:chart-lipschitz} also bounds the full conditional Hellinger
distance, with the same four coordinate differences and a changed
constant.  This step includes the nonresponse atom.

The interpolated densities used in the proof need not correspond to a
common affine regression parameter in the original coefficient box.
They are valid auxiliary densities with a controlled positive variance.
All actual statistical candidates used below remain in the original
parameter class.  The proof establishes $C^1$ regularity in the cubic
coordinate; it makes no claim of two derivatives in $q$ or a uniformly
invertible Fisher information matrix.

\section{Finite coordinates and local packing}\label{app:embedding}

Use the coordinates and map $F:\Theta\to\R^K$ defined in
\eqref{eq:overview-coordinates}--\eqref{eq:overview-F}.
The norm of $F$ is the ordinary Euclidean norm.  For symmetric matrices,
the Euclidean norm of $\operatorname{vech}$ lies between the Frobenius
norm divided by $\sqrt2$ and the Frobenius norm.  Matrix estimates below
can therefore use the Frobenius norm without changing the conclusion.

\begin{lemma}[Two metric comparisons]\label{lem:embedding}
There are constants $0<c_E\le C_E<\infty$, depending only on the fixed
model class, such that every $\theta,\theta'\in\Theta$ satisfies
\begin{equation}\label{eq:embed-two-sided}
 c_E\|F(\theta)-F(\theta')\|_2
 \le h_\nu(\theta,\theta')
 \le C_E\|F(\theta)-F(\theta')\|_2.
\end{equation}
The map $F$ is continuous and one-to-one.  For any common actual design
density $f$ in the bounded-density class, the two corresponding bounds
for $H_f$ follow by multiplying the left constant by $\sqrt\ell$ and
the right constant by $\sqrt u$.
\end{lemma}

\begin{proof}
Throughout this proof $\Delta$ denotes an unprimed quantity minus its
primed counterpart, $\delta=h_\nu(\theta,\theta')$, and $\|\cdot\|_{L^2(\nu)}$
is taken under the uniform reference law, not the unknown actual
design law.  Lemmas~\ref{lem:moments} and \ref{lem:projection}, applied
with $f=1$, give
\begin{align}
 &\|\Delta\gamma\|_2+\|\Delta t\|_\infty
 +\|\Delta M\|_{L^2(\nu)}+\|\Delta V\|_{L^2(\nu)}
 +|a^3-a'^3|\le C\delta,\label{eq:embed-initial-bounds}\\
 &|a-a'|\|g\|_2^2\le C\delta.\label{eq:embed-U-prebound}
\end{align}
Here $M=m+a\lambda(t)$ and $V=\sigma^2-a^2h(t)$ are the respondent
mean and variance.  These estimates do not hold nuisance coordinates
fixed.  In particular,
\[
 \Delta U_2=(a-a')gg^T+a'(gg^T-g'g'^T),
\]
and boundedness of $a',g,g'$ shows that
\begin{equation}\label{eq:embed-U-bound}
 \|\Delta U_2\|_{\mathrm F}\le C\delta.
\end{equation}

An independent variance calculation is needed for $Z$; the preceding
bound on $U_2$ alone does not give it.  The variance identity yields
\begin{equation}\label{eq:embed-variance-difference}
 \Delta V+a'^2\{h(t)-h(t')\}
 =\Delta\sigma^2-(a^2-a'^2)h(t).
\end{equation}
Let $\Pi_0$ be the $L^2(\nu)$ projection onto constants.  If $g\ne0$,
choose $j$ with $|g_j|\ge\|g\|_2/\sqrt d$.  Conditional on the other
coordinates, integration by parts on $[-1,1]$ gives
\begin{equation}\label{eq:embed-first-projection}
 \mathbb{E}_\nu[X_jh(t)]
 =\frac{g_j}{2}\mathbb{E}_\nu[(1-X_j^2)h'(t)],
 \qquad
 |\mathbb{E}_\nu[X_jh(t)]|\ge\frac{\kappa_*|g_j|}{3}.
\end{equation}
The boundary term vanishes because $1-X_j^2$ is zero at both
endpoints; the lower bound uses $h'=-\kappa\le-\kappa_*$ and
$\mathbb{E}_\nu(1-X_j^2)=2/3$.  Since $\mathbb{E}_\nu X_j=0$ and
$\|X_j\|_{L^2(\nu)}=1/\sqrt3$, Cauchy--Schwarz gives
\begin{equation}\label{eq:embed-variance-residual}
 \|(I-\Pi_0)h(t)\|_{L^2(\nu)}
 \ge\frac{\kappa_*}{\sqrt{3d}}\|g\|_2.
\end{equation}
The same inequality is trivial if $g=0$.  The left side of
\eqref{eq:embed-variance-difference} has $L^2(\nu)$ norm at most
$C\delta$, by \eqref{eq:embed-initial-bounds} and boundedness of $h'$.
Projection onto the orthogonal complement of constants therefore gives
\[
 |a^2-a'^2|\|g\|_2\le C\delta.
\]
Consequently,
\begin{equation}\label{eq:embed-Z-bound}
 \Delta Z=(a^2-a'^2)g+a'^2(g-g'),\qquad
 \|\Delta Z\|_2\le C\delta.
\end{equation}

Taylor's integral formula yields exact, global identities on the cube:
\begin{align}
 M(x)&=\eta_0+\eta_s^Tx+U_2:B_\gamma(x),
 &B_\gamma(x)&=xx^T\int_0^1(1-u)\kappa(t_0+ug^Tx)\,du,
 \label{eq:embed-mean-expansion}\\
 V(x)&=v_0-Z^TC_\gamma(x),
 &C_\gamma(x)&=x\int_0^1h'(t_0+ug^Tx)\,du.
 \label{eq:embed-variance-expansion}
\end{align}
Here $A:B=\sum_{j,k}A_{jk}B_{jk}$.  The arguments of $\kappa$ and $h'$
remain in a fixed bounded interval, and $\|x\|_2\le\sqrt d$.
Their bounded first derivatives, together with
\[
 |(t_0+ug^Tx)-(t'_0+ug'^Tx)|
 \le\sqrt{d+1}\|\gamma-\gamma'\|_2,
\]
show directly that $B_\gamma,C_\gamma$ are uniformly bounded and
uniformly Lipschitz in $\gamma$, in supremum norm over the cube.

Subtract \eqref{eq:embed-mean-expansion} at the two endpoints and use
\eqref{eq:embed-initial-bounds}, \eqref{eq:embed-U-bound}, and the
just established bounds for $B_\gamma$.  The affine function
$\Delta\eta_0+\Delta\eta_s^Tx$ has $L^2(\nu)$ norm at most $C\delta$.
The exact Gram identity
\[
 \|c_0+c_s^Tx\|_{L^2(\nu)}^2=c_0^2+\|c_s\|_2^2/3
\]
then controls $|\Delta\eta_0|+\|\Delta\eta_s\|_2$ by $C\delta$.
Subtracting \eqref{eq:embed-variance-expansion} and using
\eqref{eq:embed-Z-bound} likewise controls $|\Delta v_0|$.  All
coordinates of $F$ have now been controlled, proving the left
inequality in \eqref{eq:embed-two-sided}.

For the other direction, the same exact expansions and boundedness of
the coordinates give
\begin{equation}\label{eq:embed-uniform-chart-distance}
 \sup_{x\in[-1,1]^d}
 \bigl\{|\Delta M(x)|+|\Delta V(x)|+|\Delta t(x)|+|\Delta q|\bigr\}
 \le C\|F(\theta)-F(\theta')\|_2.
\end{equation}
Every original endpoint has the positive convolution variance margin
\eqref{eq:chart-model-margin}, uniformly in $x$.  Lemma~\ref{lem:chart}
and \eqref{eq:chart-observed-root} therefore bound each conditional
observed-law Hellinger distance by the left side of
\eqref{eq:embed-uniform-chart-distance}.  Integrating its square under
$\nu$ proves the right inequality in \eqref{eq:embed-two-sided}.
The chart's local interpolation argument and distant-endpoint bound
justify using it here for arbitrary two endpoints.

Finally, continuity of $F$ follows from its formulas.  Its inverse on
its image is explicit: $q$ gives the unique real $a$, $\gamma$ gives
$t_0,g$, $\eta_0,\eta_s$ give $\beta_0,\beta_s$, and $v_0$ gives
$\sigma^2=v_0+a^2h(t_0)$.  The positive root gives $\sigma$ and then
$\rho=a/\sigma$.  Thus $F$ is one-to-one.  The comparisons with a
common $f$ follow from $\sqrt\ell\,h_\nu\le H_f\le\sqrt u\,h_\nu$.
\end{proof}

\begin{lemma}[Deterministic covering and local packing]\label{lem:packing}
For every $\varepsilon>0$, there is a deterministic finite ordered set
$\mathcal{N}_\varepsilon=\{\theta_1,\ldots,\theta_N\}\subset\Theta$
that is $\varepsilon$ separated in $h_\nu$ (distinct distances exceed
$\varepsilon$) and covers $\Theta$ with radius $2\varepsilon$.
For every center $\theta\in\Theta$ and $R\ge0$,
\begin{equation}\label{eq:embed-local-packing}
 \#\{i:h_\nu(\theta_i,\theta)\le R\}
 \le\left(1+\frac{2C_E R}{c_E\varepsilon}\right)^K.
\end{equation}
The same local count holds for any finite $\varepsilon$ separated set
in a fixed compact submodel.
\end{lemma}

\begin{proof}
We first give a deterministic way to obtain the covering slack in the
full rectangular class.  There is a constant $L_0\ge1$ such that
\begin{equation}\label{eq:embed-original-lipschitz}
 h_\nu(\theta,\theta')\le L_0\|\theta-\theta'\|_2.
\end{equation}
To check this upper bound without an observed information inverse,
augment the observations, at fixed $x$, to $(Y,L)$, where $L=t(x)+U$.
This pair is Gaussian with mean $(m(x),t(x))$ and covariance
\[
 \Sigma=\begin{pmatrix}\sigma^2&\sigma\rho\\\sigma\rho&1\end{pmatrix}.
\]
Its determinant is at least
$\sigma_-^2(1-\rho_+^2)>0$, and its trace is at most
$\sigma_+^2+1$.  Hence both eigenvalues have fixed positive
lower and finite upper bounds.  The derivative of a Gaussian log
density along a unit parameter direction is the sum of a bounded
constant, a linear polynomial in the centered pair, and a quadratic
polynomial in that pair, with bounded coefficients.  Gaussian fourth
moments therefore bound its second moment uniformly in $x$ and the
parameter direction.  Since a square-root derivative equals one half
of the square root times this score, its $L^2$ norm is uniformly
bounded.  Polynomial times Gaussian envelopes justify differentiation
and integration on the parameter segment, exactly as for ordinary
derivatives in Appendix~\ref{app:chart}.  The original parameter box
is convex, so the segment stays in this uniformly nondegenerate set.
Integrating the derivative gives the complete-data version of
\eqref{eq:embed-original-lipschitz}.  The map taking $(Y,L)$ to the
response atom when $L\le0$ and to $(1,Y)$ when $L>0$ is independent of
the parameter.  Hellinger distance cannot increase under this map:
Cauchy--Schwarz on each fiber increases affinity.  Integrating under
$\nu$ proves \eqref{eq:embed-original-lipschitz}.

Choose a finite rectangular grid $G_\varepsilon\subset\Theta$ so that
every parameter lies within Euclidean distance $\varepsilon/L_0$ of
a grid point.  In a fixed lexicographic order, retain the first point
and thereafter retain a point exactly when its $h_\nu$ distance from
every previously retained point exceeds $\varepsilon$.  The resulting
set is $\varepsilon$ separated.  Each rejected grid point lies within
$\varepsilon$ of a retained point, and the original grid approximates
any parameter within $\varepsilon$ in $h_\nu$.  The triangle inequality
gives covering radius $2\varepsilon$.  This construction and its
strict-distance convention are fixed before data are observed.

For points in an $h_\nu$ ball of radius $R$, the first comparison in
\eqref{eq:embed-two-sided} places their $F$ images inside the Euclidean
ball of radius $R/c_E$ about $F(\theta)$.  The second comparison makes
their pairwise Euclidean distances exceed $\varepsilon/C_E$.  Thus
Euclidean balls of radius $\varepsilon/(2C_E)$ about those images are
disjoint and all lie in the ball of radius
$R/c_E+\varepsilon/(2C_E)$.  Comparison of $K$ dimensional volumes
gives \eqref{eq:embed-local-packing}; the images need not have interior
in $\mathbb{R}^K$ for these ambient balls to be used.

For a compact constrained submodel, a finite maximal separated set
exists by compactness and continuity of the metric.  Fix any such
set and its ordering independently of the data.  Maximality gives
covering radius $\varepsilon$, and the same ambient volume argument
gives the identical local bound.  Rectangular zero-coordinate or sign
constraints can instead use the explicit grid construction above.
\end{proof}

\section{Observable tests and the finite-net estimator}\label{app:testing}

This appendix proves Theorem~\ref{thm:estimator}. The only geometric
inputs are the metric comparison in Lemma~\ref{lem:embedding} and its
local packing consequence, Lemma~\ref{lem:packing}. All conditional
integrals below use the dominating measure $\mu$ on the disjoint union
of the nonresponse atom and the respondent line. In particular, no
test uses an unobserved outcome. The actual design density $f$ is fixed
within each experiment and is common to every law being compared.

\subsection{A conditionally centered pairwise test}

Use the candidates, statistic $Z_{ij}=\psi_{ij}-c_{ij}$, pair orientation,
and tie rule from Section~\ref{sec:estimator-construction}. For $k=i,j$,
write $e_k(x)=\int Q_k(o\mid x)\psi_{ij}(x,o)\,d\mu(o)$, so
$c_{ij}=(e_i+e_j)/2$ in \eqref{eq:overview-test}.
Both conditional branches are positive, making the denominator positive.
Continuity of the candidate densities and integration of the bounded
measurable integrand give Borel statistics. Their evaluation uses only
candidate laws and the observed $X$, without the design density $f$.

\begin{lemma}[Robust observable pairwise tests]\label{lem:pairtest}
For every common $f\in\mathcal F_{\ell,u}$, let
$d_{ij}=H_f(\theta_i,\theta_j)$. If a true parameter $s\in\Theta$
satisfies $H_f(s,\theta_i)\le d_{ij}/32$, then the test
\eqref{eq:test-definition} satisfies
\begin{equation}
 \Prob_s^{f,\otimes n}\{i\text{ loses to }j\}
 \le \exp(-n d_{ij}^2/400).
 \label{eq:test-robust}
\end{equation}
The corresponding assertion holds with $i$ and $j$ interchanged.
\end{lemma}

\begin{proof}
Write $h_x^2=\int(\sqrt{Q_j}-\sqrt{Q_i})^2\,d\mu$.
Factoring $Q_j-Q_i$ gives the exact conditional identity
\begin{equation}
 e_j(x)-e_i(x)
 =\int(Q_j-Q_i)\psi_{ij}\,d\mu=h_x^2.
 \label{eq:test-mean-identity}
\end{equation}
Consequently, after integration against the same $f\,d\nu$,
$\E_i^f Z_{ij}=-d_{ij}^2/2$ and
$\E_j^f Z_{ij}=d_{ij}^2/2$. The conditional centering is essential:
it makes these identities valid for every allowed $f$, without placing
$f$ in the statistic.

We have $|\psi_{ij}|\le1$, $|c_{ij}|\le1$, and $|Z_{ij}|\le2$.
For $k=i,j$,
\[
 \E_k[\psi_{ij}^2\mid X=x]
 =\int Q_k\frac{(\sqrt{Q_j}-\sqrt{Q_i})^2}
 {(\sqrt{Q_j}+\sqrt{Q_i})^2}\,d\mu\le h_x^2.
\]
Jensen's inequality yields
$c_{ij}(x)^2\le(e_i(x)^2+e_j(x)^2)/2\le h_x^2$.
Thus
\begin{equation}
 \E_i^f Z_{ij}^2\le4d_{ij}^2.
 \label{eq:test-second-moment}
\end{equation}
For the rest of the proof all expectations use the common actual $f$.
Set $e=H_f(s,\theta_i)$ and $Z=Z_{ij}$. Expanding
$\sqrt{dP_s^f}=\sqrt{dP_i^f}
+(\sqrt{dP_s^f}-\sqrt{dP_i^f})$ and using $|Z|\le2$ gives
\begin{align}
 \E_s Z^2&\le2\E_i Z^2+8e^2
 \le8d_{ij}^2+8e^2,\notag\\
 |\E_sZ-\E_iZ|
 &\le e\{2(\E_sZ^2+\E_iZ^2)\}^{1/2}
 \le e(24d_{ij}^2+16e^2)^{1/2}
 \le5ed_{ij}+4e^2.
 \label{eq:test-perturbation}
\end{align}
The expectation difference is bounded by Cauchy--Schwarz applied to
$Z(\sqrt{dP_s^f}-\sqrt{dP_i^f})\allowbreak
(\sqrt{dP_s^f}+\sqrt{dP_i^f})$; thus the argument covers approximate
truths and not only truths at net centers. If $e\le d_{ij}/32$, its
upper bound in \eqref{eq:test-perturbation} is at most
$(5/32+4/1024)d_{ij}^2<d_{ij}^2/4$. Therefore
\[
 \E_s Z\le-d_{ij}^2/4,\qquad
 \operatorname{Var}_s Z\le(8+8/1024)d_{ij}^2,
 \qquad |Z-\E_sZ|\le4.
\]
The event that $i$ loses, including a possible tie assigned to $j$,
is contained in $\{\sum_r Z_r\ge0\}$. Bernstein's inequality for
independent, centered variables bounded in absolute value by $4$ gives
\begin{align*}
 \Prob_s^{f,\otimes n}\{i\text{ loses to }j\}
 &\le\exp\left\{-\frac{n(d_{ij}^2/4)^2}
 {2\{(8+8/1024)d_{ij}^2+4(d_{ij}^2/4)/3\}}\right\}\\
 &=\exp\left\{-\frac{n d_{ij}^2}{3203/12}\right\}
 \le\exp(-n d_{ij}^2/400).
\end{align*}
Here $32(8+8/1024+1/3)=3203/12<400$ is an exact
constant calculation. Reversing the orientation proves the other
assertion, with the same allowance for ties.
\end{proof}

\subsection{The estimator and covering slack}

The rule in \eqref{eq:overview-estimator} is a finite-model $T_0$-estimator
with reference metric $h_\nu$ in the sense of
\citet[Section 4.2, equations (4.5)--(4.6)]{Birge2006}; Section 5 there
develops its risk analysis. Here Lemmas~\ref{lem:pairtest}
and~\ref{lem:packing} supply the tests and local metric bounds.

Use the ordered net $\mathcal N_\varepsilon$ from
Section~\ref{sec:finite-coordinates}, with $\varepsilon=n^{-1/2}$.
All candidates are legal original parameters; the auxiliary density-chart
interpolations are not candidates. In \eqref{eq:overview-packing} one may
take $C_P=2C_E/c_E$. Write
$\widehat i=\min\operatorname*{arg\,min}_{1\le i\le N}D_i$, so
$\widehat\theta_n=\theta_{\widehat i}$ as in
\eqref{eq:overview-estimator}. Finitely many Borel tests and comparisons
make this finite-valued estimator measurable.

For a true $\theta$, choose, only for the analysis, a center
$\theta_{i_0}$ with $h_\nu(\theta,\theta_{i_0})\le2\varepsilon$.
Then $H_f(\theta,\theta_{i_0})\le2\sqrt u\,\varepsilon$.
Whenever
\begin{equation}
 h_\nu(\theta_{i_0},\theta_j)
 \ge64\sqrt{u/\ell}\,\varepsilon,
 \label{eq:test-slack}
\end{equation}
this approximation error is at most
$H_f(\theta_{i_0},\theta_j)/32$. Lemma~\ref{lem:pairtest} therefore
implies
\begin{equation}
 \Prob_\theta^{f,\otimes n}\{j\text{ defeats }i_0\}
 \le\exp\{-n\ell h_\nu(\theta_{i_0},\theta_j)^2/400\}.
 \label{eq:test-center-error}
\end{equation}
The factor in \eqref{eq:test-slack} retains both the
$2\varepsilon$ covering slack and the density comparison constants.

\subsection{Proof of the risk and slope bounds}

\begin{proof}[Proof of Theorem~\ref{thm:estimator}]
In the unique comparison between $\widehat i$ and $i_0$, one candidate
loses, unless they coincide. If $\widehat i$ loses, their distance is
at most $D_{\widehat i}\le D_{i_0}$; if $i_0$ loses, it is at most
$D_{i_0}$ directly. Thus in all cases
\begin{equation}
 h_\nu(\widehat\theta_n,\theta_{i_0})\le D_{i_0},
 \qquad
 h_\nu(\widehat\theta_n,\theta)\le2\varepsilon+D_{i_0}.
 \label{eq:test-deterministic}
\end{equation}
Let $R\ge64\sqrt{u/\ell}\,\varepsilon$ and group potential defeating
candidates into shells with distances in
$[2^kR,2^{k+1}R)$, $k\ge0$, from $\theta_{i_0}$.
Using \eqref{eq:test-packing}, \eqref{eq:test-center-error}, and a union
bound gives
\begin{equation}
 \Prob_\theta^{f,\otimes n}\{D_{i_0}>R\}
 \le\sum_{k\ge0}
 (1+C_P2^{k+1}R/\varepsilon)^K
 \exp\{-n\ell4^kR^2/400\}.
 \label{eq:test-shell}
\end{equation}
Empty shells cause no difficulty. Put $R=v\varepsilon$ and
$c_0=\ell/400$. Because $K$ and $C_P$ are fixed, there is a fixed
$v_0\ge64\sqrt{u/\ell}$ such that
\[
 K\log(1+2C_Pw)\le c_0w^2/2\qquad(w\ge v_0).
\]
Apply this with $w=2^kv$ in each shell. For $v\ge v_0$,
\begin{align}
 \Prob_\theta^{f,\otimes n}\{D_{i_0}>v/\sqrt n\}
 &\le\sum_{k\ge0}e^{-c_04^kv^2/2}\notag\\
 &\le\frac{e^{-c_0v^2/2}}{1-e^{-3c_0v^2/2}}
 \le A_0e^{-b_0v^2},
 \label{eq:test-no-log}
\end{align}
where $4^k\ge1+3k$ and $A_0,b_0>0$ depend only on the fixed class.
This local count is the reason no $\log n$ factor occurs; a bound on
the total number of candidates would not give the same conclusion.

Equation~\eqref{eq:test-deterministic}, followed by a fixed shift in
$v$, yields constants $A_h,b_h,v_h>0$ such that
\begin{equation}
 \sup_{\theta,f}
 \Prob_\theta^{f,\otimes n}
 \{h_\nu(\widehat\theta_n,\theta)>v/\sqrt n\}
 \le A_he^{-b_hv^2},\qquad v\ge v_h.
 \label{eq:test-reference-tail}
\end{equation}
For example, replace $v$ in \eqref{eq:test-no-log} by $v-2$ and
take $v\ge\max(4,2v_0+2)$, so that $(v-2)^2\ge v^2/4$.
Since $H_f\le\sqrt u\,h_\nu$, a further fixed rescaling gives
the Hellinger tail in the theorem for every $n\ge1$. Thresholds larger
than the model diameter have zero tail probability.

If that tail has constants $A,b,v_*$, integration of its square gives
\[
 \E_\theta^{f,\otimes n}H_f^2(\widehat\theta_n,\theta)
 =\int_0^\infty
 \Prob_\theta^{f,\otimes n}\{H_f^2(\widehat\theta_n,\theta)>t\}\,dt
 \le\frac{v_*^2}{n}
 +\int_{v_*^2/n}^\infty Ae^{-bnt}\,dt\le\frac Cn.
\]
Finally, Theorem~\ref{thm:modulus} supplies
$\|\widehat\beta_s-\beta_s\|\le C_M
H_f(\widehat\theta_n,\theta)^{2/3}$. Hence the event
$\|\widehat\beta_s-\beta_s\|>v n^{-1/3}$ implies
$H_f(\widehat\theta_n,\theta)>(v/C_M)^{3/2}n^{-1/2}$.
For $v$ above a fixed threshold, the Hellinger tail gives
$A'\exp(-b'v^3)$, as asserted. Every constant is uniform in the
allowed common $f$ and in the true parameter.
\end{proof}

\section{Support separation and simultaneous confidence}\label{app:support}

\subsection{Proof of the support frontier}

For laws $P,Q$ dominated by a common measure, define their affinity
by $\mathcal A(P,Q)=\int\sqrt{dP\,dQ}=1-H^2(P,Q)/2$.
Product factorization and Cauchy--Schwarz give
\begin{equation}
 \mathcal A(P^{\otimes n},Q^{\otimes n})
 =\{1-H^2(P,Q)/2\}^n,
 \qquad
 \operatorname{TV}(P,Q)\le\sqrt{1-\mathcal A(P,Q)^2}.
 \label{eq:support-affinity}
\end{equation}
The second bound applies Cauchy--Schwarz to
$|p-q|=|\sqrt p-\sqrt q|(\sqrt p+\sqrt q)$.
A test choosing $Q$ on $B$ has error sum
$1-[Q(B)-P(B)]\ge1-\operatorname{TV}(P,Q)$; integrating over
any additional randomization gives the same bound for randomized tests.

\begin{proof}[Proof of Theorem~\ref{thm:support}]
For sufficiency, apply Theorem~\ref{thm:estimator} and
Theorem~\ref{thm:modulus} to obtain fixed $A,c,C_0>0$ such that
\begin{equation}
 \sup_{\theta\in\Theta,f}
 \Prob_\theta^{f,\otimes n}
 \{\|\widehat\beta_s-\beta_s\|>t\}
 \le Ae^{-cnt^3},\qquad nt^3\ge C_0.
 \label{eq:support-slope-tail}
\end{equation}
For the prespecified positive sequence $b_n$, use the rule \eqref{eq:threshold-rule}.
If $\theta\in\Theta(b_n)$ and
$\|\widehat\beta_s-\beta_s\|\le b_n/4$, then a truly zero
coordinate has estimate of magnitude at most $b_n/4$, and a nonzero
coordinate has estimate of magnitude at least $3b_n/4$. Thus
\eqref{eq:support-threshold} recovers the entire support on this
event. The margin handles ties and atoms in the estimator's
distribution without any continuity assumption. Consequently
\begin{equation}
 \sup_{f}\sup_{\theta\in\Theta(b_n)}
 \Prob_\theta^{f,\otimes n}
 \{\widehat S_n(b_n)\ne S(\theta)\}
 \le Ae^{-c'nb_n^3},\qquad nb_n^3\ge C_1.
 \label{eq:support-upper}
\end{equation}
The right side tends to zero whenever $n^{1/3}b_n\to\infty$.
The vector error bound already handles all $d$ coordinates, so there
is no additional dimension-dependent union bound that grows with $n$.

For necessity, suppose $n^{1/3}b_n$ does not tend to infinity.
There is an infinite subsequence $I$ and a finite $B>0$ such that
$n^{1/3}b_n\le B$ for $n\in I$. In particular $b_n\to0$ along
$I$. Use Assumption~\ref{ass:natural} and the two laws
of Theorem~\ref{thm:hardpath}, with
\[
 r_n=(b_n/h_0)^{1/2},\qquad h_0=h(t_0)>0.
\]
Eventually along $I$, both endpoints are in the fixed class; their
supports are $\varnothing$ and $\{1\}$, and the sole nonzero outcome
slope is exactly $h_0r_n^2=b_n$. Thus both endpoints belong to
$\Theta(b_n)$, including the endpoint with empty support. Their
scale changes and correlations remain feasible, and both correlations
are strictly positive. For any fixed allowed $f$, the distance bound
from that theorem implies
\begin{equation}
 H_{f,n}^2:=H_f^2(\theta_{1,r_n},\theta_{2,r_n})
 \le C_0'b_n^3,\qquad nH_{f,n}^2\le C_0'B^3.
 \label{eq:support-path-bound}
\end{equation}
Constants can be chosen uniformly over $f$; in particular the
reference design $f=1$ is sufficient for a uniform impossibility
result.

Eventually $H_{f,n}^2\le1$. The elementary inequality
$\log(1-z/2)\ge-z$, $0\le z\le1$, and
\eqref{eq:support-affinity} show that the product affinity satisfies
\[
 \mathcal A_n\ge e^{-nH_{f,n}^2}\ge e^{-C_0'B^3}>0.
\]
Any support rule induces a test: decide for the second law exactly
when the estimated support is $\{1\}$. Each test error is bounded
by the corresponding support error. Its maximum error is at least
half the sum lower bound, so
\begin{align}
 \max_{k=1,2}\Prob_{\theta_{k,r_n}}^{f,\otimes n}
 \{\widehat S_n\ne S(\theta_{k,r_n})\}
 &\ge\frac{1-\sqrt{1-\mathcal A_n^2}}2\notag\\
 &\ge\frac{\mathcal A_n^2}4
 \ge\frac14 e^{-2C_0'B^3}.
 \label{eq:support-lower}
\end{align}
Taking the infimum over all support rules proves the lower bound on
the subsequence, even if those rules know $f$. This works for every
finite $B$, not only a sufficiently small multiplier. The conclusion
is that uniform exact recovery on the full sequence is impossible;
it does not assert a positive full-sequence lower limit for every
oscillating $b_n$.
\end{proof}

If $b_n=o(n^{-1/3})$, the same proof has $nH_{f,n}^2\to0$,
$\mathcal A_n\to1$ and product total variation tending to zero.
The two-point lower bound for the maximal support error then tends to
$1/2$. This additional statement also requires the feasible path
condition. The argument needs only the Hellinger path bound, not a
Kullback--Leibler expansion.

The order of quantifiers matters. For each specified $b_n$ in the
recoverable regime, the rule \eqref{eq:support-threshold} is permitted
to use it. Choosing once and for all a threshold
$\tau_n=L_n n^{-1/3}$ with $L_n\to\infty$ gives a different
assertion and would require enough separation relative to $\tau_n$.
The proof above does not give a single such fixed divergent $L_n$
that handles every unknown, arbitrarily slowly diverging
$n^{1/3}b_n$, nor an impossibility theorem for all conceivable adaptive
procedures.

\subsection{A secondary confidence consequence}

The following corollary follows from the estimator and inverse modulus;
it does not require support consistency.

\begin{corollary}[Simultaneous confidence and slope diameter]
\label{cor:confidence}
Let $0<\alpha<1$, and use constants $A_h,b_h,v_h$ from
\eqref{eq:test-reference-tail}. Choose $v_\alpha\ge v_h$ such that
$A_he^{-b_hv_\alpha^2}\le\alpha$. Then
\begin{equation}
 C_{n,\alpha}=\{\vartheta\in\Theta:
 h_\nu(\vartheta,\widehat\theta_n)\le v_\alpha/\sqrt n\}
 \label{eq:confidence-region}
\end{equation}
has coverage at least $1-\alpha$ uniformly in $\theta,f$ and has
deterministic slope-projection diameter at most
$C v_\alpha^{2/3}n^{-1/3}$. At truths satisfying either fixed-strength
condition in Proposition~\ref{prop:strong}, its slope-projection
diameter is $O_{\Prob}(n^{-1/2})$, uniformly over those truths and $f$,
for each fixed $\alpha$.
The same uniform $O_{\Prob}(n^{-1/2})$ slope-diameter conclusion
holds over $\theta\in\Theta_0$ from
Corollary~\ref{cor:zero-selection-slope}, including $\rho=0$,
for each fixed $\alpha$.

If the feasible weak-neighborhood condition holds and
$0<\alpha<1/2$, there are constants $c_\alpha,\delta_\alpha>0$
such that, for every sufficiently large $n$, any measurable confidence
region with coverage at least $1-\alpha$ throughout $\Theta$ has,
at a parameter in $\Theta$, slope-projection diameter at least
$c_\alpha n^{-1/3}$ with probability at least $\delta_\alpha$.
The latter assertion already holds for a fixed allowed $f$.
\end{corollary}

\begin{proof}
The estimator belongs to \eqref{eq:confidence-region}, so the region
is nonempty. The event that it fails to contain the truth is precisely
$\{h_\nu(\theta,\widehat\theta_n)>v_\alpha/\sqrt n\}$, whose
probability is at most $\alpha$. If $\vartheta,\vartheta'$ are
both in the region, their $h_\nu$ distance is at most
$2v_\alpha/\sqrt n$. Theorem~\ref{thm:modulus}, applied with the
reference design, yields the deterministic diameter bound.
Measurability causes no additional issue: $h_\nu$ is continuous in
its parameters, $\Theta$ is compact, and the center has finitely
many possible values, so each resulting region and its diameter are
measurable.

For a fixed-strength true parameter $\theta$, every
$\vartheta\in C_{n,\alpha}$ satisfies
\[
 h_\nu(\vartheta,\theta)
 \le v_\alpha/\sqrt n+h_\nu(\widehat\theta_n,\theta).
\]
The one-strong-endpoint inequality in Proposition~\ref{prop:strong}
therefore bounds its slope distance from the true slope by a fixed
multiple of the right side. Taking twice the supremum and using
\eqref{eq:test-reference-tail} proves the asserted uniform
$O_{\Prob}(n^{-1/2})$ diameter bound. This argument applies on the
whole sample space, not only on the coverage event.

For $\theta\in\Theta_0$, the same triangle inequality and
\eqref{eq:zero-selection-slope-modulus} give, on the whole sample space,
\[
 \diam\{\beta_s(\vartheta):\vartheta\in C_{n,\alpha}\}
 \le 2L_0\{v_\alpha/\sqrt n+
                h_\nu(\widehat\theta_n,\theta)\}.
\]
Equation~\eqref{eq:test-reference-tail} proves the claimed uniform
$O_{\Prob}(n^{-1/2})$ bound. This argument does not require the truth
to lie in $C_{n,\alpha}$ or the candidates to lie in $\Theta_0$.

For the lower bound fix an allowed $f$, take the path of
Theorem~\ref{thm:hardpath}, and set $r=c n^{-1/6}$ for a constant
$c>0$ to be chosen. The product affinity argument gives, for all
sufficiently large $n$,
\[
 \operatorname{TV}(P_{\theta_{1,r}}^{f,\otimes n},
                     P_{\theta_{2,r}}^{f,\otimes n})
 \le\sqrt{1-e^{-2C_Hc^6}}.
\]
Because $1-2\alpha>0$, choose $c$ small enough that this upper
bound is strictly less than $1-2\alpha$, and write their positive
difference as $\delta_\alpha$. If $\widetilde C_n$ covers each
endpoint with probability at least $1-\alpha$ under its own law,
the total variation inequality transfers coverage of the second
endpoint to the first law with a loss of at most that distance.
A union bound under the first law then gives
\[
 \Prob_{\theta_{1,r}}^{f,\otimes n}
 \{\theta_{1,r},\theta_{2,r}\in\widetilde C_n\}
 \ge1-2\alpha-
 \operatorname{TV}(P_{\theta_{1,r}}^{f,\otimes n},
                     P_{\theta_{2,r}}^{f,\otimes n})
 \ge\delta_\alpha.
\]
On this event the projected diameter is at least
$h_0r^2=h_0c^2n^{-1/3}$. Taking
$c_\alpha=h_0c^2$ proves the claim. The restriction $\alpha<1/2$
is needed for this particular two-point lower bound, although the
coverage and diameter upper bounds hold for all $0<\alpha<1$.
\end{proof}

The full-parameter coverage event protects all coordinate projections
at once, including projections chosen for reporting after looking at
the data. This is an unconditional simultaneous statement. It does not
assert coverage conditional on a coordinate having been selected.
If $\alpha=\alpha_n\to0$, the displayed construction can take
$v_{\alpha_n}$ of order $\sqrt{\log(1/\alpha_n)}$, leading to
a slope diameter of order
$\{\log(1/\alpha_n)/n\}^{1/3}$; the confidence cost should not be
discarded. Generic uniform-confidence principles near singular
information have antecedents discussed in the main text.

\subsection{Proofs of the slope corollaries}\label{app:slope-corollaries}

\begin{proof}[Proof of Corollary~\ref{cor:minimax-slope-risk}]
Theorem~\ref{thm:modulus}, Jensen's inequality for the concave
function $x\mapsto x^{2/3}$, and \eqref{eq:main-risk} give
\[
 \E_\theta^{f,\otimes n}
   \|\widehat\beta_{s,n}-\beta_s\|^2
 \le C\E_\theta^{f,\otimes n}
       \{H_f^2(\widehat\theta_n,\theta)\}^{2/3}
 \le C\{\E_\theta^{f,\otimes n}
       H_f^2(\widehat\theta_n,\theta)\}^{2/3}
 \le Cn^{-2/3}
\]
uniformly in $f,\theta$.

For the lower bound, fix the allowed design $f=1$. Choose $r_0>0$
and $C_H>0$ such that the path of Theorem~\ref{thm:hardpath} is
feasible and has $H_1^2\le C_Hr^6$ for $0<r\le r_0$.
Take a fixed $c_0>0$ with $c_0\le r_0$ and
$C_Hc_0^6\le1/4$, and put $r_n=c_0n^{-1/6}$.
The two slopes are separated by
$\Delta_n=h_0c_0^2n^{-1/3}$ for every $n\ge1$.
Write $P_n,Q_n$ for the corresponding product laws.
Under our Hellinger convention, affinity factorization yields
\[
 H^2(P_n,Q_n)
 =2\bigl[1-\{1-H_1^2/2\}^n\bigr]\le nH_1^2,
 \qquad
 \TV(P_n,Q_n)\le H(P_n,Q_n)
 \le\sqrt{C_H}\,c_0^3\le\tfrac12.
\]
Here $H_1$ is evaluated at the two one-observation path endpoints.
Any slope estimator induces the test that selects the nearer
endpoint slope, with a fixed tie rule. Each testing error forces
squared estimation error at least $\Delta_n^2/4$; the sum of
testing errors is at least $1-\TV(P_n,Q_n)$. Therefore its maximum
risk over these endpoints is at least
\[
 \frac{\Delta_n^2}{8}\{1-\TV(P_n,Q_n)\}
 \ge\frac{h_0^2c_0^4}{16}n^{-2/3}.
\]
This proves the lower bound after taking the infimum over all
observed-data estimators.

The full-model estimator gives the upper bound on either sign class.
The positive path and its reflection in the proof of
Proposition~\ref{prop:sign} give the same lower bound whenever the
respective half-neighborhood is feasible.
\end{proof}

\begin{proof}[Proof of Corollary~\ref{cor:zero-selection-slope}]
For $\theta\in\Theta_0$, the index $t(x)=\gamma_0$ is constant.
The function $B$ in \eqref{eq:geom-B} is therefore affine, with slope
$\beta_s-\beta'_s$. Lemma~\ref{lem:moments}, applied with the
reference design, bounds $\norm B_{L^2(\nu)}$ by a fixed multiple of
$h_\nu(\theta,\theta')$. The affine Gram matrix yields
$\norm{\beta_s-\beta'_s}\le\sqrt3\norm B_{L^2(\nu)}$;
\eqref{eq:design-comparison} proves
\eqref{eq:zero-selection-slope-modulus}.
Apply that inequality to $\theta'=\widehat\theta_n$ and use
\eqref{eq:test-reference-tail} for the tail bound.
The squared-risk bound follows from
\eqref{eq:main-risk} and
$\norm{\widehat\beta_{s,n}-\beta_s}^2
 \le L_0^2\ell^{-1}H_f^2(\widehat\theta_n,\theta)$.
All constants are uniform and the estimator bounds hold for every
$n\ge1$.
\end{proof}

\begin{proof}[Proof of Corollary~\ref{cor:transition-upper}]
Fix the true endpoint $\theta$ and an arbitrary competitor
$\theta'\in\Theta$. Put $H=H_f(\theta,\theta')$ and
$v=|a-a'|$. Lemmas~\ref{lem:moments} and~\ref{lem:projection} give
\[
 v^3\le CH,\qquad vG^2\le CH,\qquad
 \|\beta_s-\beta_s'\|\le C(H+vG).
\]
For $G>0$, these imply the two bounds
$vG\le CGH^{1/3}$ and $vG\le CH/G$. Thus
\[
 \|\beta_s-\beta_s'\|^2
 \le C\left[H^2+
          \min\{G^2H^{2/3},G^{-2}H^2\}\right].
\]
The constants depend only on the fixed model bounds, not on $G$.
Take $\theta'=\widehat\theta_n$. By \eqref{eq:main-risk},
$\E H^2\le C/n$. Moreover, Jensen's inequality gives
$\E H^{2/3}\le(\E H^2)^{1/3}\le Cn^{-1/3}$.
Since the expectation of the minimum of two nonnegative random
variables is at most the minimum of their expectations,
\eqref{eq:transition-risk} follows.
For $G=0$, the displayed slope bound is directly $CH$.
Finally, retain the first term inside the minimum and use
$G\le Mn^{-1/3}$ to obtain \eqref{eq:shrinking-layer-risk}.
\end{proof}

\section{Fixed-strength and scientifically restricted regimes}
\label{app:restrictions}

These comparisons retain free scale unless stated otherwise. Constants
may depend on fixed strengths. For a strength-dependent upper bound
near $g=0$, see Corollary~\ref{cor:transition-upper}.

\begin{proposition}[Knowing the correlation sign]\label{prop:sign}
Restrict the free-scale class to $\rho\ge0$, or to $\rho\le0$.
The outcome-slope upper bound of Theorem~\ref{thm:modulus} and the
statistical upper bounds of Theorem~\ref{thm:estimator} continue to
hold. If the corresponding sign half-neighborhood in
Assumption~\ref{ass:natural} is feasible, the slope modulus still
has order $\delta^{2/3}$, and for each prespecified $b_n>0$ the sharp
criterion for uniform exact support recovery is
$n^{1/3}b_n\to\infty$. The lower bounds can use two strictly MNAR
endpoints with the specified common sign.
\end{proposition}

\begin{proof}
Restriction gives the deterministic and full-model statistical upper
bounds. An estimator valued in the sign class follows by constructing
the grid within the closed half-box, with the same local packing and
observable tests.

For sufficiently small $r>0$, Theorem~\ref{thm:hardpath} gives
$a_1=r$, $a_2=2r$, $\rho_1,\rho_2>0$, $H_f\asymp r^3$, and
slope gap $h_0r^2$. For the negative class, replace $a_k,m_k$ by
$-a_k,-m_k$ and keep $t_r,\sigma_k$ fixed. This common reflection
$y\mapsto-y$ preserves Hellinger distance and the nonresponse atom;
the supports remain empty and singleton, with second slope $-h_0r^2$.
The respective half-neighborhood ensures feasibility. Taking
$r\asymp\delta^{1/3}$ proves the modulus bound, and the threshold
and product-affinity arguments of Theorem~\ref{thm:support} give
the support criterion. Excluding $\rho=0$ without bounding $|\rho|$
away from zero changes neither conclusion: the paths are strict-sign,
and the closed-class estimator remains valid at every strict-sign truth.
\end{proof}

Free scale permits the compensation
$\sigma_k^2=\sigma_0^2+k^2r^2h_0$ in respondent variance;
knowing the sign does not prevent it.

\begin{proposition}[One fixed-strength endpoint]\label{prop:strong}
Let $g_0>0$ or $a_0>0$ be fixed. If $\theta$ satisfies
$\|\gamma_s\|\ge g_0$ or $|a|=|\sigma\rho|\ge a_0$, then,
for every $\theta'\in\Theta$,
\begin{equation}
 \|\theta-\theta'\|\le C H_f(\theta,\theta').
 \label{eq:strong-inverse}
\end{equation}
The constant is uniform in $f$ and in the respective strong truths;
the competing endpoint is unrestricted within the original class.
In particular, the same full-class estimator as in
Theorem~\ref{thm:estimator} satisfies
\begin{equation}
 \sup_{f,\theta\ {\rm strong}}
 \Prob_\theta^{f,\otimes n}
 \{\|\widehat\beta_s-\beta_s\|>t\}
 \le A e^{-cnt^2},\qquad nt^2\ge C_0,
 \label{eq:strong-tail}
\end{equation}
and has squared slope risk at most $C/n$ on these sets of truths.
\end{proposition}

\begin{proof}
Put $D=H_f(\theta,\theta')$, $v=|a-a'|$, and $G=\|\gamma_s\|$.
Lemmas~\ref{lem:moments}--\ref{lem:projection} and the proof of
Theorem~\ref{thm:modulus} give
\begin{equation}
 \begin{split}
 |a^3-a'^3|&\le CD,\qquad vG^2\le CD,
 \qquad\|\gamma-\gamma'\|\le CD,\\
 \|B\|_{L^2(\nu)}&\le CD,
 \qquad B=m-m'+(a-a')\lambda(t),\qquad
 \|V-V'\|_{L^2(\nu)}\le CD.
 \end{split}
 \label{eq:strong-input}
\end{equation}
If $G\ge g_0$, the second bound gives $v\le Cg_0^{-2}D$.
If $|a|\ge a_0$, use the factorization
\[
 |a^3-a'^3|=v(a^2+aa'+a'^2)
 =v\{(a'+a/2)^2+3a^2/4\}
 \ge(3a_0^2/4)v
\]
to obtain $v\le CD$ again, even when the two correlations have
different signs. Since $\lambda(t)$ is uniformly bounded,
$\|m-m'\|_{L^2(\nu)}\le\|B\|_{L^2(\nu)}+Cv\le CD$.
The affine Gram identity
$\|c_0+c_s^Tx\|_{L^2(\nu)}^2=c_0^2+\|c_s\|^2/3$
then controls all components of $\beta-\beta'$.

The variance identity gives
\[
 \sigma^2-\sigma'^2
 =V-V'+(a^2-a'^2)h(t)+a'^2\{h(t)-h(t')\}.
\]
The right side has $L^2(\nu)$ norm at most $CD$, using
$|a^2-a'^2|\le2A v$, the bounded derivative of $h$, and the
selection-index bound. The left side is constant, so the lower
scale bound gives $|\sigma-\sigma'|\le CD$. Finally
$\rho=a/\sigma$ implies $|\rho-\rho'|\le CD$.
This proves \eqref{eq:strong-inverse} with only the first endpoint strong.

Set $\theta'=\widehat\theta_n$ in \eqref{eq:strong-inverse}
and apply Theorem~\ref{thm:estimator} to obtain
\eqref{eq:strong-tail}, including weak estimator outputs.
Integration, bounding the probability by one for $nt^2<C_0$,
gives squared risk at most $C/n$.
\end{proof}

The marginal selection slope $\gamma_s$ differs from the conditional
coefficient in
\[
 \Prob(R=1\mid X=x,Y=y)=\Phi(\alpha_0+\alpha_s^Tx+\alpha_y y),
\]
since
\begin{equation}
 \alpha_y=\frac{\rho}{\sigma\sqrt{1-\rho^2}},\qquad
 \alpha_s=\frac{\gamma_s-\rho\beta_s/\sigma}
 {\sqrt{1-\rho^2}}.
 \label{eq:conditional-selection-coefficients}
\end{equation}
The two strength conditions are not numerically identical, but compactness
gives $\|\alpha_s\|\le C(\|\gamma_s\|+|a|)$. Hence
$\|\alpha_s\|\ge c_0>0$ implies a fixed lower bound on either
$\|\gamma_s\|$ or $|a|$; Proposition~\ref{prop:strong} on their
union gives root-$n$ accuracy. Likewise, $|\rho|\ge\rho_0>0$
implies $|a|\ge\sigma_-\rho_0$.

\begin{proposition}[A continuous exclusion restriction]
\label{prop:exclusion}
Let the full design be $(X,W)\in[-1,1]^{d+d_W}$ with the same
fixed upper and positive lower bounds on its joint density, where
$d_W\ge1$. Suppose the known outcome restriction is
$m(x,w)=\beta_0+\beta_s^Tx$, so that $\beta_W=0$, and let
$t(x,w)=\gamma_0+\gamma_X^Tx+\gamma_W^Tw$.
For two endpoints sharing this known zero restriction, if the first
has $\|\gamma_W\|\ge g_W>0$, then
\begin{equation}
 |a-a'|\,\|\gamma_W\|\le C H_f(\theta,\theta'),
 \qquad \|\theta-\theta'\|\le C_{g_W}H_f(\theta,\theta').
 \label{eq:exclusion-inverse}
\end{equation}
A tournament constructed within the known zero-constraint model has
the root-$n$ tail \eqref{eq:strong-tail} for the target $X$ slopes.
The full affine-model tournament also has that target-slope tail at
these truths, even when its competing candidates have nonzero
outcome $W$ slopes.
\end{proposition}

\begin{proof}
The moment calculation in total dimension $d+d_W$ gives
\[
 B(x,w)=m(x)-m'(x)+(a-a')\lambda(t(x,w)),
 \qquad\|B\|_{L^2(\nu)}\le CD,
 \quad D=H_f(\theta,\theta').
\]
Choose $j$ such that
$|\gamma_{W,j}|\ge\|\gamma_W\|/\sqrt{d_W}$ and set
$\psi(z)=(1-z^2)^2$ and $\zeta_j(w)=\psi'(w_j)$.
For the uniform reference measure,
\[
 \E_\nu\zeta_j=0,\qquad
 \|\zeta_j\|_{L^2(\nu)}^2=128/105,
 \qquad\E_\nu\psi(W_j)=8/15.
\]
Because both outcome means are independent of $W$,
$\langle m-m',\zeta_j\rangle_\nu=0$.
An integration by parts in $w_j$, with
$\psi(-1)=\psi(1)=0$, yields
\begin{align*}
 \langle\lambda(t),\zeta_j\rangle_\nu
 &=-\gamma_{W,j}\E_\nu[\lambda'(t)\psi(W_j)]\\
 &=\gamma_{W,j}\E_\nu[h(t)\psi(W_j)].
\end{align*}
Since $h_*:=\min_{|t|\le T}h(t)>0$, its absolute value is at
least $(8/15)h_*|\gamma_{W,j}|$. Cauchy--Schwarz therefore gives
the explicit bound
\[
 |a-a'|\,\|\gamma_W\|
 \le\frac{15\sqrt{d_W}}{8h_*}\sqrt{\frac{128}{105}}
       \|B\|_{L^2(\nu)}\le CD.
\]
This integration is with respect to $\nu$, not the possibly
nonsmooth actual $f\nu$; the norm comparison with $H_f$ was
already included in the moment bound. The fixed strength implies
$|a-a'|\le C_{g_W}D$, and the affine mean, variance, and scale
steps in the proof of Proposition~\ref{prop:strong} then prove the
second inequality in \eqref{eq:exclusion-inverse}.

The known-zero model is a fixed coordinate slice of the full box.
The grid on this slice retains Lemma~\ref{lem:packing}, and the
embedding and conditional tests restrict to it, giving the restricted
estimator.

For the full-model tournament,
$\|\gamma_s^{\rm total}\|=\|(\gamma_X,\gamma_W)\|\ge g_W$.
Proposition~\ref{prop:strong} in dimension $d+d_W$ gives root-$n$
accuracy against unrestricted competitors, hence for the target $X$
slopes, without cancelling their affine $W$ terms.
\end{proof}

The projection requires the shared zeros $\beta_W=\beta'_W=0$;
nonzero competing $W$ slopes destroy its cancellation. This is a
restriction of that proof, not a universal necessity: fixed nonzero
$\gamma_W$ permits the separate full-model argument above. Under exclusion,
$\gamma_W=\sqrt{1-\rho^2}\alpha_W$, so the fixed correlation bound
makes marginal and conditional exclusion strengths comparable.
Naming an exclusion variable or assuming marginal continuity alone
is insufficient.

\begin{corollary}[Support in fixed-strength subclasses]
\label{cor:regular-support}
Let $\Theta_*$ be one of the fixed-strength subclasses just described,
and define its separated class by $\Theta_*(b_n)=\Theta_*\cap
\Theta(b_n)$, selecting only its allowed target slopes. For each
prespecified $b_n$, if $\sqrt n\,b_n\to\infty$, thresholding at
$b_n/2$ gives uniform support consistency with error at most
$Ae^{-cnb_n^2}$ for $nb_n^2$ above a fixed constant.

For a matching converse, assume additionally that $\Theta_*$ contains
a parameter $\theta_0$ with all target slopes zero and an allowed
target coordinate $j$, such that replacing only $\beta_j=0$ by
$\beta_j=z$ remains in $\Theta_*$ for every sufficiently small
$z\ge0$. The other outcome slopes and the known exclusion zeros,
if present, remain zero. Then uniform exact support recovery on
$\Theta_*(b_n)$ is possible if and only if
$\sqrt n\,b_n\to\infty$, and its slope modulus has order $\delta$.
\end{corollary}

\begin{proof}
The threshold margin \eqref{eq:support-threshold} gives exact support,
including the empty support, when target-vector error is at most $b_n/4$.
Apply \eqref{eq:strong-tail} at $t=b_n/4$.

For the converse use the stipulated feasible slope path, fixing
$a,\sigma,t$ and changing the target outcome slope by $z$.
The complete variables $(Y,L)$, where $L=t+U$, have, conditional on
the design, common covariance
\[
 \Sigma=\begin{pmatrix}\sigma^2&\sigma\rho\\
                         \sigma\rho&1\end{pmatrix}
\]
and mean difference $(zx_j,0)^T$. Its smallest eigenvalue has a
fixed positive lower bound. For normals with a common covariance,
their squared Hellinger distance is
$2-2\exp\{-\Delta\mu^T\Sigma^{-1}\Delta\mu/8\}$, which is
at most $\Delta\mu^T\Sigma^{-1}\Delta\mu/4\le Cz^2$.
The observation map sends $(Y,L)$ to the missing atom for $L\le0$
and to the respondent outcome for $L>0$. It is parameter independent,
so Hellinger distance contracts under it. Integration over the
common design law gives $H_f^2(P_z,P_0)\le Cz^2$ uniformly in
$f$. This also proves a direct slope-modulus lower bound of order
$\delta$, by taking $z=\delta/\sqrt C$; the matching upper bound
is \eqref{eq:strong-inverse}.

If $\sqrt n\,b_n$ fails to diverge, choose an infinite subsequence
on which it is at most a finite $B$. The paths with $z=b_n$
are feasible eventually, have supports $\varnothing$ and $\{j\}$,
and belong to the exact separated subclass. They satisfy
$nH_f^2(P_{b_n},P_0)\le CB^2$. Equations
\eqref{eq:support-affinity}--\eqref{eq:support-lower}, with
$Cb_n^2$ in place of $C_0'b_n^3$, bound the maximal support error
below by $\tfrac14e^{-2CB^2}>0$ on that subsequence. Feasible perturbations are necessary for this argument; smaller
subclasses inherit only the upper bound. For restrictions on
$\alpha_s$, a strict baseline strength margin preserves the
restriction under sufficiently small slope perturbations.
\end{proof}

\begin{remark}[Known scale]\label{rem:known-scale}
The weak-selection class with fixed known $\sigma$ inherits the
$O_{\Prob}(n^{-1/3})$ slope upper bound and the sufficient support
condition $b_n\gg n^{-1/3}$ by restriction of the main theorems.
It does not contain the free-scale hard path, whose two scales move
differently. The sharp weak-selection modulus, statistical rate, and
support boundary for known scale are not determined here. The
fixed-strength root-$n$ results remain valid in that subclass.
\end{remark}

\section{Protocol and checks for the two-point experiment}\label{app:numerical}
\subsection{Fixed design and observed-data test}
The experiment uses the two endpoints in \eqref{eq:hardpath}, with $X$ uniform on $[-1,1]$, $t_0=0$, and $\sigma_0=1.3$. We use the fixed rectangular class
\[
 \beta_0,\beta_1,g\in[-1,1],\qquad \gamma_0\in[-0.5,0.5],\qquad
 \sigma\in[1,2],\qquad \rho\in[-0.9,0.9].
\]
This class contains the natural weak neighborhood. Every simulated endpoint is strictly inside these bounds, with $\rho>0$ and $\sigma^2(1-\rho^2)\ge\sigma_0^2/2$. The response probability conditional on $X=x$ is $\Phi(r_nx)$, so the overall response probability is exactly $1/2$ for both endpoints.

The choices $n\in\{512,2048,8192,32768\}$ and
\[
 r_n=0.6(n/512)^{-\alpha},\qquad
 \alpha\in\{1/4,1/6,1/12\},\qquad b_n=h(0)r_n^2
\]
were fixed before simulation, as were the master seed and replication count. Since $nb_n^3$ is proportional to $n^{1-6\alpha}$, these paths lie below, at, and above the separation order. These labels concern orders, not optimal leading constants. The three paths share their $n=512$ endpoints, giving ten distinct testing problems; the common baseline is generated once.

At endpoint $k$, independently draw $X\sim\mathrm{Unif}[-1,1]$ and independent standard normal $\varepsilon,Z$, set $U=\rho_k\varepsilon+\sqrt{1-\rho_k^2}Z$, and generate $(Y,R)$ from \eqref{eq:latent}. The test uses only the resulting observed sample. With $q_k$ denoting the respondent component in \eqref{eq:respondent}, its statistic is
\begin{equation}\label{eq:two-point-lr}
 T_n=\sum_{i:R_i=1}\log\frac{q_2(Y_i\mid X_i)}{q_1(Y_i\mid X_i)}.
\end{equation}
The common design density and identical nonresponse masses cancel. No unobserved response enters \eqref{eq:two-point-lr}. The rule selects endpoint 2 when $T_n\ge0$. Define $e_1=P_1^{\otimes n}(T_n\ge0)$, $e_2=P_2^{\otimes n}(T_n<0)$ and $e=(e_1+e_2)/2$. This is a simple-versus-simple, equal-prior testing problem with both endpoints supplied.

\subsection{Monte Carlo uncertainty and reproducibility}
Each distinct pair uses $M=2000$ independent data sets per endpoint.
The frozen design, independent random streams, runtime versions, error counts,
and all 40000 sample log-likelihood ratios are documented in the ancillary
\path{NUMERICAL_PROTOCOL_DETAILS.txt} and reproducibility files
(see \hyperref[sec:code-data]{Code and data availability}).

Writing $\widehat e_k$ for the observed error proportion at endpoint $k$, the reported average and estimated standard error are
\[
 \widehat e=\frac{\widehat e_1+\widehat e_2}{2},\qquad
 \widehat{\mathrm{se}}(\widehat e)=
 \left\{\frac{\widehat e_1(1-\widehat e_1)+\widehat e_2(1-\widehat e_2)}{4(M-1)}\right\}^{1/2}.
\]
We also form two-sided exact binomial intervals $[l_k,u_k]$ of confidence level $97.5\%$ separately for the two endpoint errors. The interval $[(l_1+l_2)/2,(u_1+u_2)/2]$ covers $e$ with probability at least $95\%$ by the union bound. These intervals are pointwise for each testing problem, not simultaneous over the scenario grid. Table~\ref{tab:two-point-results} includes every prespecified scenario. The estimated standard errors range from $0.00361$ to $0.00781$.

\renewcommand{\thetable}{H.\arabic{table}}
\setcounter{table}{0}
\begin{table}[tbp]
\centering
\caption{Known-endpoint discrimination on the prespecified paths. Error rates are proportions. The final column gives conservative pointwise Monte Carlo intervals with coverage at least $95\%$.}
\label{tab:two-point-results}
\normalsize
\setlength{\tabcolsep}{4pt}
\begin{tabular}{@{}lrrrrrrr@{}}
\toprule
Path & $n$ & $r_n$ & $nb_n^3$ & $\widehat e_1$ & $\widehat e_2$ & $\widehat e$ & MC interval \\
\midrule
Below & 512 & 0.600 & 6.163 & 0.277 & 0.310 & 0.293 & [0.271, 0.317] \\
 & 2048 & 0.424 & 3.082 & 0.339 & 0.359 & 0.349 & [0.325, 0.373] \\
 & 8192 & 0.300 & 1.541 & 0.394 & 0.392 & 0.393 & [0.368, 0.418] \\
 & 32768 & 0.212 & 0.770 & 0.420 & 0.426 & 0.423 & [0.398, 0.448] \\
\addlinespace
Boundary & 512 & 0.600 & 6.163 & 0.277 & 0.310 & 0.293 & [0.271, 0.317] \\
 & 2048 & 0.476 & 6.163 & 0.273 & 0.296 & 0.285 & [0.262, 0.308] \\
 & 8192 & 0.378 & 6.163 & 0.283 & 0.284 & 0.284 & [0.261, 0.307] \\
 & 32768 & 0.300 & 6.163 & 0.281 & 0.290 & 0.286 & [0.264, 0.309] \\
\addlinespace
Above & 512 & 0.600 & 6.163 & 0.277 & 0.310 & 0.293 & [0.271, 0.317] \\
 & 2048 & 0.535 & 12.327 & 0.210 & 0.222 & 0.216 & [0.196, 0.238] \\
 & 8192 & 0.476 & 24.653 & 0.131 & 0.122 & 0.126 & [0.110, 0.144] \\
 & 32768 & 0.424 & 49.307 & 0.055 & 0.055 & 0.055 & [0.044, 0.068] \\
\bottomrule
\end{tabular}
\end{table}

\subsection{Distance checks and interpretation}
The convention in \eqref{eq:hellinger} gives single-observation affinity $1-H_1^2/2$. For the product laws let $A_n=(1-H_1^2/2)^n$. The likelihood-ratio rule minimizes the equal-prior error, so $e=\{1-\TV(P_1^{\otimes n},P_2^{\otimes n})\}/2$. The inequalities $1-A_n\le\TV\le\sqrt{1-A_n^2}$ yield
\begin{equation}\label{eq:two-point-affinity}
 \frac{1-\sqrt{1-A_n^2}}{2}\le e\le\frac{A_n}{2}.
\end{equation}
Figure~\ref{fig:two-point} evaluates these bounds using Gauss--Legendre quadrature for $X$ and Gauss--Hermite quadrature for the response integral. All ten pairs were checked by doubling node counts, verifying unit mass, and directly evaluating affinity; detailed diagnostics are recorded in the ancillary \path{NUMERICAL_PROTOCOL_DETAILS.txt}.

The coefficient in \eqref{eq:path-limit} is $L(0,1.3)=0.0133482867$. For $r=0.1,0.05,0.025$, the computed ratios $H_1^2/(Lr^6)$ are $1.00458$, $1.00120$, and $1.00030$. All reported average-error intervals intersect the corresponding affinity bounds; no simulated sample produced a likelihood-ratio tie.

\section{Geometry comparison and nonuniform-design experiments}\label{app:additional-experiments}
\renewcommand{\thetable}{I.\arabic{table}}
\setcounter{table}{0}
\renewcommand{\thefigure}{I.\arabic{figure}}
\setcounter{figure}{0}

\subsection{Construction of the matched-gap comparison}
The design for Figure~\ref{fig:geometry-comparison} uses normalized uniform measure on $[-1,1]$, $t_0=0$, $\sigma_0=1.3$, and 20 equally spaced values of $\log r$ from $r=0.025$ to $r=0.3$. Write $\delta=h(0)r^2$. The first pair is exactly \eqref{eq:hardpath}. For each of the other two pairs, set
\[
 \begin{gathered}
 (\beta_{0,1},\beta_{1,1})=(0,0),\qquad
 (\beta_{0,2},\beta_{1,2})=(0,\delta),\\
 \gamma_{0,k}=g_k=0,\quad \sigma_k=1.3,\quad \rho_k=a/1.3.
 \end{gathered}
\]
with $a=0$ or $a=0.4$ held fixed across the two endpoints. Thus $g=0$ does not require $\rho=0$. Every pair lies inside the rectangular class in Appendix~\ref{app:numerical}. Its two nonresponse atoms are equal pointwise in $x$, so their contribution to $H_f^2$ is exactly zero.

We integrate the actual respondent subdensities using the Hellinger convention in \eqref{eq:hellinger}. The weak path has limiting normalization $H_f^2/\delta^3\to L(0,1.3)/h(0)^3=0.05173508$. For the fixed-nuisance zero-layer paths, let $q_a(y)$ be the respondent subdensity at outcome location zero, and define $J_a=\int q_a(y)\{\partial_y\log q_a(y)\}^2\,dy$. The location expansion and $E_\nu X^2=1/3$ give $H_f^2/\delta^2\to J_a/12$. For $a=0$, this is $1/(24\sigma_0^2)=0.02465483$; for $a=0.4$, numerical integration gives $0.02623677$. The former also has the exact independent check
\[
 H_f^2=\int_{-1}^1\left\{1-\exp\left(-\frac{\delta^2x^2}{8\sigma_0^2}\right)\right\}\,d\nu(x).
\]
All 60 path--grid combinations were checked by doubling the outer quadrature nodes, refining the inverse-Mills-ratio integral, and directly subtracting log-CDF values. The normal-location identity and observed-law masses were also checked. Node counts and discrepancy bounds are recorded in the ancillary \path{NUMERICAL_PROTOCOL_DETAILS.txt}.

All 19 consecutive log--log slopes per path are retained, with no selected fitting window. Their ranges are $[3.00035,3.01055]$ on the weak path, $[1.99989,2.00000]$ for $a=0$, and $[1.99988,2.00000]$ for $a=0.4$, after rounding. The files \texttt{geometry.csv} and \texttt{geometry\_local\_slopes.csv} contain all 60 quadrature rows and 57 consecutive-interval slopes.

\subsection{Nonuniform covariate designs and simulation protocol}
We retain the endpoints, sample sizes, parameter class, and three $r_n$ paths of Appendix~\ref{app:numerical}, and replace the uniform design by
\[
 f_{\mathrm{tilt}}(x)=1+0.6x,\qquad
 f_{\mathrm{edge}}(x)=0.5+1.5x^2,\qquad -1\le x\le1.
\]
These are densities relative to $\nu$, not Lebesgue measure. Both integrate to one; their ranges are $[0.4,1.6]$ and $[0.5,2]$, respectively. Hence both belong to $\mathcal F_{0.4,2}$. Their means are $0.2$ and $0$, and their second moments are $1/3$ and $7/15$. The first design breaks symmetry and the second puts more weight near the edges. The conditional response probability remains $\Phi(r_nx)$. Its design average equals $1/2$ for the symmetric edge-weighted design and exceeds $1/2$ for the tilted design; it is computed separately for every pair.

Exact inverse-CDF sampling, its checks, and the random-stream specification
are recorded in the ancillary \path{NUMERICAL_PROTOCOL_DETAILS.txt}.
Outcomes, selection, and the observed-data test follow Appendix~\ref{app:numerical}.
The follow-up design was fixed after the uniform experiment and before new
outcomes were generated; it was not externally preregistered. With ten distinct
pairs per design and $M=2000$ independent repetitions per endpoint, the reused
$n=512$ pair yields 24 displayed scenarios, 20 distinct pairs, and 80000
saved likelihood ratios. Every scenario is retained.

\subsection{Complete testing results and Monte Carlo intervals}
In addition to $e=(e_1+e_2)/2$, we report $e_{\max}=\max(e_1,e_2)$ for the specified likelihood-ratio rule. From the endpoint $97.5\%$ Clopper--Pearson intervals $[l_k,u_k]$, the interval $[\max(l_1,l_2),\max(u_1,u_2)]$ covers $e_{\max}$ with probability at least $95\%$, by the same union-bound argument used for the average. These guarantees are pointwise for each testing problem. The maximum is over the two supplied endpoints only; the equal-prior likelihood-ratio rule need not minimize this maximum.

Figure~\ref{fig:nonuniform} and Tables~\ref{tab:nonuniform-tilt}--\ref{tab:nonuniform-edge}
display every scenario. Both designs reproduce the uniform experiment's
qualitative separation pattern. The tables report endpoint-specific, average,
and maximum errors and intervals; the ancillary CSV retains unrounded values.

\begin{figure}[tbp]
\centering
\includegraphics[width=\linewidth]{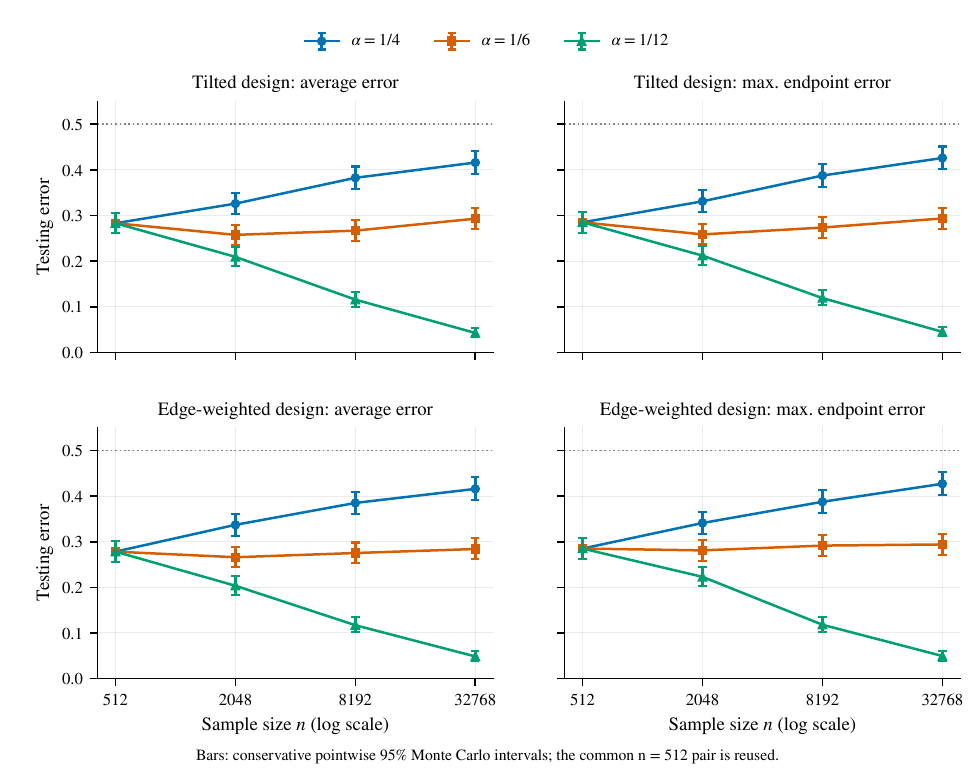}
\caption{Complete nonuniform-design results. Rows correspond to the two design densities; columns show estimates of equal-prior average error and of the maximum error over the two supplied endpoints. All three parameter paths and four sample sizes are displayed. Bars are conservative pointwise $95\%$ Monte Carlo intervals. The dotted line marks $1/2$ as a reference level, not an upper bound for the maximum endpoint error.}
\label{fig:nonuniform}
\end{figure}

\begin{table}[tbp]
\centering
\caption{Tilted design, $f(x)=1+0.6x$: all testing scenarios. Intervals are conservative pointwise $95\%$ Monte Carlo intervals.}
\label{tab:nonuniform-tilt}
\small\setlength{\tabcolsep}{3.5pt}
\begin{tabular}{crrrrrrr}
\toprule
$\alpha$ & $n$ & $\widehat e_1$ & $\widehat e_2$ & $\widehat e$ & Interval for $e$ & $\widehat e_{\max}$ & Interval for $e_{\max}$ \\
\midrule
$1/4$ & 512 & 0.285 & 0.281 & 0.283 & [0.261, 0.306] & 0.285 & [0.263, 0.308] \\
$1/4$ & 2048 & 0.331 & 0.321 & 0.326 & [0.303, 0.350] & 0.331 & [0.307, 0.355] \\
$1/4$ & 8192 & 0.388 & 0.378 & 0.383 & [0.358, 0.407] & 0.388 & [0.363, 0.412] \\
$1/4$ & 32768 & 0.426 & 0.406 & 0.416 & [0.391, 0.441] & 0.426 & [0.401, 0.451] \\
\addlinespace
$1/6$ & 512 & 0.285 & 0.281 & 0.283 & [0.261, 0.306] & 0.285 & [0.263, 0.308] \\
$1/6$ & 2048 & 0.259 & 0.257 & 0.258 & [0.236, 0.280] & 0.259 & [0.237, 0.281] \\
$1/6$ & 8192 & 0.260 & 0.274 & 0.267 & [0.245, 0.290] & 0.274 & [0.251, 0.296] \\
$1/6$ & 32768 & 0.292 & 0.293 & 0.293 & [0.270, 0.316] & 0.293 & [0.271, 0.317] \\
\addlinespace
$1/12$ & 512 & 0.285 & 0.281 & 0.283 & [0.261, 0.306] & 0.285 & [0.263, 0.308] \\
$1/12$ & 2048 & 0.206 & 0.212 & 0.209 & [0.189, 0.230] & 0.212 & [0.192, 0.233] \\
$1/12$ & 8192 & 0.112 & 0.119 & 0.115 & [0.100, 0.132] & 0.119 & [0.103, 0.136] \\
$1/12$ & 32768 & 0.040 & 0.045 & 0.042 & [0.033, 0.054] & 0.045 & [0.035, 0.057] \\
\bottomrule
\end{tabular}
\end{table}

\begin{table}[tbp]
\centering
\caption{Edge-weighted design, $f(x)=0.5+1.5x^2$: all testing scenarios. Intervals are conservative pointwise $95\%$ Monte Carlo intervals.}
\label{tab:nonuniform-edge}
\small\setlength{\tabcolsep}{3.5pt}
\begin{tabular}{crrrrrrr}
\toprule
$\alpha$ & $n$ & $\widehat e_1$ & $\widehat e_2$ & $\widehat e$ & Interval for $e$ & $\widehat e_{\max}$ & Interval for $e_{\max}$ \\
\midrule
$1/4$ & 512 & 0.272 & 0.285 & 0.278 & [0.256, 0.301] & 0.285 & [0.263, 0.308] \\
$1/4$ & 2048 & 0.333 & 0.341 & 0.337 & [0.313, 0.361] & 0.341 & [0.317, 0.365] \\
$1/4$ & 8192 & 0.388 & 0.383 & 0.385 & [0.361, 0.410] & 0.388 & [0.363, 0.412] \\
$1/4$ & 32768 & 0.427 & 0.405 & 0.416 & [0.391, 0.441] & 0.427 & [0.402, 0.452] \\
\addlinespace
$1/6$ & 512 & 0.272 & 0.285 & 0.278 & [0.256, 0.301] & 0.285 & [0.263, 0.308] \\
$1/6$ & 2048 & 0.251 & 0.281 & 0.266 & [0.244, 0.289] & 0.281 & [0.259, 0.304] \\
$1/6$ & 8192 & 0.259 & 0.291 & 0.275 & [0.253, 0.298] & 0.291 & [0.269, 0.315] \\
$1/6$ & 32768 & 0.294 & 0.274 & 0.284 & [0.262, 0.307] & 0.294 & [0.271, 0.317] \\
\addlinespace
$1/12$ & 512 & 0.272 & 0.285 & 0.278 & [0.256, 0.301] & 0.285 & [0.263, 0.308] \\
$1/12$ & 2048 & 0.184 & 0.223 & 0.204 & [0.184, 0.224] & 0.223 & [0.202, 0.245] \\
$1/12$ & 8192 & 0.116 & 0.118 & 0.117 & [0.101, 0.134] & 0.118 & [0.102, 0.135] \\
$1/12$ & 32768 & 0.048 & 0.050 & 0.049 & [0.038, 0.060] & 0.050 & [0.039, 0.061] \\
\bottomrule
\end{tabular}
\end{table}

\subsection{Numerical diagnostics and reproducibility}
Quadrature weights include the design density relative to $\nu$. For all 20 distinct pairs, checks compare doubled node counts, unit mass, and direct affinity with $1-H_f^2/2$. All average-error intervals intersect the design-specific bounds in \eqref{eq:two-point-affinity}; no likelihood-ratio tie occurs. Simulated mean log-likelihood ratios agree with quadrature within $2.01$ Monte Carlo standard errors. Detailed diagnostics are recorded in \path{NUMERICAL_PROTOCOL_DETAILS.txt}.

Table~\ref{tab:nonuniform-checks} gives $nH_f^2$ and theoretical and empirical
response rates for every scenario. The ancillary CSV also includes exact
error counts, $r_n$, $b_n$, $nb_n^3$, endpoint intervals, affinity bounds,
mean-statistic checks, and timings. Reproduction commands are provided in
the ancillary \path{REPRODUCIBILITY.txt}.

\begin{table}[!htbp]
\centering
\caption{Distance and response-rate checks for all nonuniform scenarios, with the two designs shown side by side. $p_R$ is the quadrature response probability; $\widehat p_{R,k}$ uses all generated observations at endpoint $k$.}
\label{tab:nonuniform-checks}
\small\setlength{\tabcolsep}{3.5pt}
\begin{tabular}{crrrrrrrrr}
\toprule
 & & \multicolumn{4}{c}{Tilted design} & \multicolumn{4}{c}{Edge-weighted design} \\
\cmidrule(lr){3-6}\cmidrule(lr){7-10}
$\alpha$ & $n$ & $nH_f^2$ & $p_R$ & $\widehat p_{R,1}$ & $\widehat p_{R,2}$ & $nH_f^2$ & $p_R$ & $\widehat p_{R,1}$ & $\widehat p_{R,2}$ \\
\midrule
$1/4$ & 512 & 0.3494 & 0.5462 & 0.5463 & 0.5461 & 0.3376 & 0.5000 & 0.5001 & 0.5000 \\
$1/4$ & 2048 & 0.1821 & 0.5333 & 0.5331 & 0.5332 & 0.1789 & 0.5000 & 0.5000 & 0.4998 \\
$1/4$ & 8192 & 0.0899 & 0.5237 & 0.5237 & 0.5240 & 0.0902 & 0.5000 & 0.4999 & 0.4998 \\
$1/4$ & 32768 & 0.0437 & 0.5168 & 0.5168 & 0.5168 & 0.0447 & 0.5000 & 0.5000 & 0.5001 \\
\addlinespace
$1/6$ & 512 & 0.3494 & 0.5462 & 0.5463 & 0.5461 & 0.3376 & 0.5000 & 0.5001 & 0.5000 \\
$1/6$ & 2048 & 0.3621 & 0.5372 & 0.5374 & 0.5373 & 0.3535 & 0.5000 & 0.5001 & 0.5004 \\
$1/6$ & 8192 & 0.3640 & 0.5297 & 0.5296 & 0.5297 & 0.3601 & 0.5000 & 0.5000 & 0.5000 \\
$1/6$ & 32768 & 0.3595 & 0.5237 & 0.5237 & 0.5236 & 0.3607 & 0.5000 & 0.4999 & 0.5000 \\
\addlinespace
$1/12$ & 512 & 0.3494 & 0.5462 & 0.5463 & 0.5461 & 0.3376 & 0.5000 & 0.5001 & 0.5000 \\
$1/12$ & 2048 & 0.7143 & 0.5415 & 0.5412 & 0.5415 & 0.6936 & 0.5000 & 0.5000 & 0.4994 \\
$1/12$ & 8192 & 1.4482 & 0.5372 & 0.5370 & 0.5373 & 1.4141 & 0.5000 & 0.4998 & 0.4999 \\
$1/12$ & 32768 & 2.9133 & 0.5333 & 0.5333 & 0.5332 & 2.8625 & 0.5000 & 0.5000 & 0.5000 \\
\bottomrule
\end{tabular}
\end{table}

\aliaslabel{eq:F}{eq:overview-F}
\aliaslabel{eq:main-embedding}{eq:embed-two-sided}

% Reuse main-text definitions in Appendix E without repeating displays.
\aliaslabel{eq:test-definition}{eq:overview-test}
\aliaslabel{eq:test-packing}{eq:overview-packing}
\aliaslabel{eq:test-tournament}{eq:overview-estimator}

% Reuse main-text coordinates and support rule in Appendices D and F.
\aliaslabel{eq:embed-coordinates}{eq:overview-coordinates}
\aliaslabel{eq:embed-F}{eq:overview-F}
\aliaslabel{eq:support-threshold}{eq:threshold-rule}

\section*{Code and data availability}\label{sec:code-data}
The proofs and numerical protocols are included in Appendices A--I of this
preprint. Programs and data for its numerical exhibits are included as ancillary
files, under \path{anc/reproducibility/} in the source archive. They contain the
fixed uniform and two nonuniform designs, all 120000 saved likelihood-ratio
statistics, the quadrature results, and the corresponding numerical checks.
The guide \path{REPRODUCIBILITY.txt} lists files, dependencies, and reproduction
commands; \path{NUMERICAL_PROTOCOL_DETAILS.txt} records random streams,
runtime versions, exact nonuniform sampling, and quadrature diagnostics.
Running \path{run_reproduction.py} in \texttt{stored} mode recomputes summaries,
quadrature, figures, and tables from the saved statistics; \texttt{fresh} mode
regenerates the Monte Carlo samples from the fixed random streams.
The scope of these experiments is stated in Section~\ref{sec:numerical}.

\begin{funding}
This research received no funding.
\end{funding}

\section*{Competing interests}
The author declares no competing interests.

\bibliographystyle{imsart-nameyear}
\bibliography{references}
\end{document}